%% file: main.tex
\documentclass[11pt,a4paper]{article}

\usepackage{kantlipsum}

\usepackage{titlesec}
\usepackage[
  top=2cm,
  bottom=2cm,
  left=3cm,
  right=3cm,
  headheight=15pt, 
  includehead,includefoot,
  heightrounded, 
]{geometry}
\newlength\FHoffset
\titleformat{\paragraph}
{\normalfont\normalsize\bfseries}{\theparagraph}{1em}{}
\titlespacing*{\paragraph}
{0pt}{1.25ex plus 1ex minus .2ex}{1.5ex plus .2ex}

\usepackage{amssymb}

\usepackage{listings}
\usepackage{mathtools}

\usepackage[english]{babel}
\usepackage{amsmath}
\usepackage{amsthm}
\usepackage{hyperref}

\usepackage[useregional]{datetime2}

\usepackage{algorithm}

\usepackage{algpseudocode}
\usepackage{tikz-cd}
\tikzset{
  symbol/.style={
    draw=none,
    every to/.append style={
      edge node={node [sloped, allow upside down, auto=false]{$#1$}}}
  }
}
\usepackage{pgfplots}
\pgfplotsset{compat=1.15}
\usepackage{mathrsfs}
\usetikzlibrary{arrows}

\usepackage{float}
\usepackage{resizegather}

\usepackage[normalem]{ulem}
\useunder{\uline}{\ul}{}
\usepackage{wrapfig}
\usepackage{lipsum}
\usepackage{graphicx}
\usepackage{enumerate}
\usepackage{amsfonts}

\usepackage{microtype}
\usepackage{nomencl}
\numberwithin{equation}{section}
\usepackage{bm}

\hypersetup{colorlinks=true,
                     urlcolor=blue,
                     citecolor=blue,
                     linkcolor=red}

\newcommand{\R}{\mathbb{R}}

\newcommand{\N}{\mathbb{N}}
\newcommand{\C}{\mathbb{C}}

\newcommand{\Z}{\mathbb{Z}}
\usepackage[english]{babel}

\usepackage{multicol}
\usepackage{upgreek}

\usepackage{parskip}
\usepackage{amssymb}
\usepackage{amsthm}
\usepackage{gensymb}
\usepackage{graphicx,wrapfig,lipsum}
\usepackage{url}
\usepackage[export]{adjustbox}
\usepackage{afterpage}
\usepackage{threeparttable}
\usepackage{float}
\usepackage{textcomp}

\usepackage{listings}
\usepackage{color} 
\definecolor{mygreen}{RGB}{28,172,0} 
\definecolor{mylilas}{RGB}{170,55,241}
\usepackage{graphicx}
\usepackage{tikz-cd}

\usepackage{eurosym}
\edef\svtheparindent{\the\parindent}
\parindent=\svtheparindent\relax
\usepackage{dcolumn}
\usepackage{bm}

\usepackage{booktabs}
\usepackage{multirow}

\usepackage{framed}

\usepackage[section]{placeins}

\usepackage{nomencl}
\usepackage{inputenc}
\usepackage{amsmath}
\usepackage[toc,page]{appendix}
\usepackage[final]{pdfpages}

\usepackage{empheq,etoolbox}

\patchcmd{\subequations}
  {\theparentequation\alph{equation}}
  {\theparentequation.\alph{equation}}
  {}{}

\usepackage[super]{nth}
\newtheorem{theorem}{Theorem}[section]
\newtheorem{lemma}[theorem]{Lemma}
\newtheorem{proposition}[theorem]{Proposition}

\newtheorem{definition}[theorem]{Definition}

\newtheorem{remark}[theorem]{Remark}

\title{Non-negative polynomials without\\ hyperbolic certificates of non-negativity}

\author{H.L. Brian Ng\thanks{Department of Electrical and Computer Systems Engineering, Monash University, Clayton VIC 3800, Australia. \texttt{hin.ng@monash.edu}} \and James Saunderson\thanks{Department of Electrical and Computer Systems Engineering, Monash University, Clayton VIC 3800, Australia. \texttt{james.saunderson@monash.edu}}}
\begin{document}
\maketitle

\begin{abstract}
In this paper we study the relationship between the set of all non-negative multivariate homogeneous polynomials and those, which we call hyperwrons, whose non-negativity can be deduced from an identity involving the Wronskians of hyperbolic polynomials. We give a sufficient condition on positive integers $m$ and $2y$  such that there are  non-negative polynomials of degree $2y$ in $m$ variables that are not hyperwrons. Furthermore, we give an explicit example of a non-negative quartic form that is not a sum of hyperwrons. We partially extend our results to hyperzouts, which are polynomials whose non-negativity can be deduced from an identity involving the B\'ezoutians of hyperbolic polynomials.

\end{abstract}

\tableofcontents
\input{Main_Body/Introduction}

\input{Main_Body/Preliminaries}

\input{Main_Body/Hyperbolic_Certificates}

\input{Main_Body/Relationship_Between_Hyperbolic_And_SOS}

\input{Main_Body/Dimension_Analysis}
\input{Main_Body/Not_Hyperwron_Example}

\input{Main_Body/Research_Outlook}

\section*{Acknowledgements} 
The authors were supported in part by an Australian Research Council Discovery Early Career Researcher Award (project number DE210101056) funded by the Australian Government. 
\bibliographystyle{plain}
\bibliography{references.bib}

\end{document}

%% file: Main_Body/Introduction.tex
\section{Introduction}
\label{section: introduction}

The problem of deciding whether a multivariate polynomial with real coefficients is non-negative is a central question in computational real algebraic geometry. The development of algorithms to certify polynomial non-negativity, and their application to polynomial optimisation, has led to new computational methods in areas such as control and dynamical systems~\cite{2005PositiveControl,Papachristodoulou2002OnDecomposition,Prajna2004StochasticCertificates}, fluid mechanics~\cite{Chernyshenko2014PolynomialAhead,Fantuzzi2022TheComputations}, game theory~\cite{Parrilo2006PolynomialOptimization} and quantum information~\cite{Doherty2004CompleteCriteria}. One way to show that a polynomial is non-negative is to write it as a sum of squares (SOS) of other polynomials. This construction immediately guarantees the non-negativity of the resulting polynomial. This is a useful sufficient condition because the problem of deciding whether a polynomial is a sum of squares can be reduced to a semidefinite programming feasibility problem~\cite{Lasserre2000GlobalMoments,Parrilo2003SemidefiniteProblems,Nesterov2000SquaredProblems,Shor1987ClassFunctions}. 

However, not all non-negative polynomials can be expressed as sum of squares of polynomials, a result due to Hilbert~\cite{Hilbert1888UeberFormenquadraten}. There are numerous ways to build more expressive families of non-negative polynomials that can be searched over via convex optimization (see Section~\ref{sec:related} for further discussion). Among these, one approach involves constructing families of non-negative polynomials out of hyperbolic polynomials, which are multivariate polynomials with real coefficients and certain real-rootedness properties. (See Section~\ref{section: hyperbolic polynomials} for a formal definition.) 

If $p$ is a hyperbolic polynomial, then associated with $p$ is a convex cone, called a hyperbolicity cone. The simplest construction of non-negative polynomials from hyperbolic polynomials is as follows. Given a hyperbolic polynomial $p$ and points $u,v$ in the associated hyperbolicity cone, the Wronskian $q(x) = D_up(x)D_vp(x) - p(x)D_{uv}^2p(x)$ is a non-negative polynomial. (Here $D_ap$ denotes the directional derivative of $p$ in the direction $a$.) Higher degree polynomials can be obtained by composition with a polynomial map $\phi$, giving non-negative polynomials of the form
\begin{equation}\label{eq:wron} D_up(\phi(x))D_vp(\phi(x)) - p(\phi(x))D_{uv}^2p(\phi(x)).\end{equation}
This construction (from~\cite{Saunderson2019CertifyingOptimization}) is discussed in more detail in Section~\ref{sec: Hyperbolic-Wronskian certificates}. 
If we fix $p$ and $u$ and $\phi$, then the problem of deciding whether a given polynomial can be expressed in the form~\eqref{eq:wron} for some $v$ in the hyperbolicity cone can be solved via hyperbolic programming, a generalization of semidefinite programming~\cite{Guler1997HyperbolicProgramming}. If a polynomial can be expressed in the form~\eqref{eq:wron} then we say that it has a \emph{hyperbolic-Wronskian certificate of non-negativity}. For brevity, in this paper we use the term \emph{hyperwron} to refer to any polynomial that has a hyperbolic-Wronskian certificate of non-negativity.

\subsection{Our contributions}

This paper is focused on understanding how the collection of hyperwrons in $m$ variables of degree $2y$, as well as certain larger families of non-negative polynomials, are related to sums of squares on the one hand, and all non-negative polynomials on the other hand. 

 If a homogeneous polynomial is a hyperwron with respect to a quadratic hyperbolic polynomial then it is the composition of the non-negative quadratic form $D_up(x)D_vp(x) - p(x)D_{uv}^2p(x)$ and a polynomial map $\phi$. The result is a sum of squares. The converse, that all sums of squares have hyperbolic-Wronskian certificates of non-negativity with respect to a degree two hyperbolic polynomial, also holds (see~\cite[Remark 1]{Shu2024AlgebraicOptimization}). We discuss this in more detail in Section~\ref{section: Relationship between hyperbolic certificates and Sum of squares certificates}. This shows that hyperwrons coincide with non-negative polynomials whenever sums of squares coincide with non-negative polynomials.
 
 Our first main result, Theorem~\ref{theorem: m_y no Wronskian}, gives a sufficient condition on the degrees $2y$, and numbers of variables $m$, for which there are non-negative homogeneous polynomials that are not hyperwrons. We show that such polynomials exist by bounding the dimension of hyperwrons that are not sums of squares, and comparing it with the dimension of all non-negative homogeneous polynomials that are not sums of squares.
\begin{theorem}
\label{theorem: m_y no Wronskian}
If $m,y$ are positive integers such that
\begin{itemize}
\item $m=4$ and $y\geq 4$ or 
\item $m= 5$ and $y\geq 3$ or
    \item $m\geq 6$ and $y\geq 2$,
\end{itemize}
then there exists a non-negative homogeneous polynomial in $m$ variables of degree $2y$ that is not a hyperwron.

\end{theorem}

Theorem~\ref{theorem: m_y no Wronskian} tells us that there are non-negative polynomials that are not hyperwrons. However, it is not clear whether the set of hyperwrons is closed under addition. In fact, we conjecture that it is not closed under addition. Therefore, given some fixed (even) degree and number of variables, it is natural to consider the conic hull of the set of hyperwrons, i.e., the larger set of \emph{sums of hyperwrons}. One can then ask whether there exist non-negative polynomials that are not sums of hyperwrons. 

As a partial answer to this question, we give an explicit example of a non-negative homogeneous quartic polynomial that is not a sum of hyperwrons. We do this by finding a non-negative homogeneous quartic that is simultaneously not a hyperwron (see Theorem~\ref{theorem: example polynomial is not a hyperwron}) and that is also extremal in the cone of non-negative quartics (Proposition~\ref{proposition: example polynomial is an extreme ray}). The example is most concisely expressed in terms of quaternions. In the statement below, if $x$ is a quaternion then  $x^*$ denotes its conjugate and $|x|^2 = xx^* = x^*x$ denotes its squared magnitude.
\begin{theorem}
\label{thm:main-sum-hyperwrons}
    Let $x,y,z,w$ be quaternion-valued indeterminates. The real-valued quartic homogeneous polynomial, $(|x|^2+|y|^2)(|z|^2+|w|^2)-|xz^*+yw^*|^2$,  in $16$ real variables, is not a sum of hyperwrons.

\end{theorem}

 Another generalization of hyperwrons can be obtained by using B\'ezoutian matrices instead of Wronskians in the initial construction. Indeed if $p$ is a hyperbolic polynomial of degree $d$ in $n$ variables, and $u,v$ are elements of the associated hyperbolicity cone, then a certain $d\times d$ parameterized B\'ezoutian matrix $B_{p,u,v}(x)$ related to $p$, $u$, and $v$, is positive semidefinite for all $x\in \R^n$. (The details of this construction are reviewed in Section~\ref{sec: Hyperbolic-Bezoutian certificates}.) Therefore, any scalar polynomial of the form
\begin{equation}\label{eq:bezoutian} \xi(x)^\intercal B_{p,u,v}(\phi(x)) \xi(x),\end{equation}
where $\phi$ and $\xi$ are suitable polynomial mappings, is non-negative. Again, if $p$ and $u$ and $\phi$ and $\xi$ are fixed, the problem of deciding whether a given polynomial can be expressed in the form~\eqref{eq:bezoutian} can be solved via hyperbolic programming~\cite{Saunderson2019CertifyingOptimization}. If a polynomial can be expressed in the form~\eqref{eq:bezoutian}, then we say that it has a \emph{hyperbolic-B\'ezoutian certificate of non-negativity}. For brevity, we use the term \emph{hyperzout} to refer to any polynomial that has a hyperbolic-B\'ezoutian certificate of non-negativity.

One key challenge with working with hyperzouts, rather than hyperwrons, is related to the degree of the hyperbolic polynomials that can arise in the associated certificates of non-negativity. It may be possible to express a hyperwron $q$ in the form~\eqref{eq:wron} for many different hyperbolic polynomials $p$, points $u,v$, and maps $\phi$. However, the degrees of $p$, $\phi$ and $q$ satisfy $\deg(q) = 2(\deg(p)-1)\deg(\phi)$, constraining the possible degrees of $p$ and $\phi$ in terms of the degree of $q$. Similarly, a hyperzout $q$ can be expressed in the form~\eqref{eq:bezoutian} in many different ways. The $(i,j)$ entry (for $0\leq i,j\leq d-1$) of the B\'ezoutian $B_{p,u,v}(x)$ appearing in~\eqref{eq:bezoutian} has
degree $2(d-1)-(i+j)$. Therefore, if the map $\xi$ only extracts the low-degree part of the matrix $B_{p,u,v}(x)$, it is possible for a hyperzout $q$ to have a representation of the form~\eqref{eq:bezoutian} where the degree of the hyperbolic polynomial $p$ is not bounded in terms of the degree of $q$. In order to generalize Theorem~\ref{theorem: m_y no Wronskian} to the hyperzout setting, we will focus on hyperzouts $q$ where the degree of the hyperbolic polynomial $p$ and the map $\phi$ involved in the certificate~\eqref{eq:bezoutian} satisfy $\deg(p)\deg(\phi) < \deg(q)$ (see Definition~\ref{definition: Bezoutian certificates} for a more precise statement). Hyperzouts satisfying the degree restriction from Definition~\ref{definition: Bezoutian certificates} are called \emph{degree-restricted hyperzouts} throughout this work.

Our main result related to degree restricted hyperzouts is Theorem~\ref{theorem: m_y no Bezoutian}. This result implies that there exist non-negative homogeneous polynomials of degree $2y\geq 4$ in $m$ variables that are not degree-restricted hyperzouts as long as $m$ is sufficiently large (for fixed $y$).

\subsection{Related Work}
\label{sec:related}

A number of families of non-negative homogeneous polynomials (also known as forms) have been studied extensively in recent years, often motivated by applications in polynomial optimization. Among these are sums of squares, sums of non-negative circuit polynomials~\cite{Iliman2016LowerProgramming,Iliman2016AmoebasCircuits} (and closely related agiforms~\cite{Reznick1989FormsInequality} and sums of AM-GM exponential polynomials~\cite{Chandrasekaran2014RelativeOptimization}), and (scaled) diagonally dominant sums of squares~\cite{Ahmadi2019DSOSOptimization}. An ongoing line of research is concerned with the relationships between these different families of non-negative forms. Hilbert's celebrated theorem~\cite{Hilbert1888UeberFormenquadraten} characterizes the degrees and numbers of variables for which non-negative forms are sums of squares. More recent refinements of this result characterize the varieties for which non-negative quadratic forms on the variety are sums of squares~\cite{Blekherman2017DoResolutions}. Sums of squares and sums of non-negative circuit polynomials are incomparable, with neither set containing the other in general~\cite{Dressler2021RealPolynomials}. With respect to the monomial basis, scaled-diagonally dominant sums of squares can be interpreted as sums of binomials squared (studied earlier in, e.g.,~\cite{Reznick1989FormsInequality}), which are sums of non-negative circuit polynomials, and so all scaled diagonally dominant sums of squares are sums of non-negative circuit polynomials. 

There are other natural properties of forms that imply non-negativity, such as being convex. Since convex forms are non-negative, in the Hilbert cases it follows that every convex form is a sum of squares. Remarkably, it also holds that convex quaternary quartic forms are always sums of squares~\cite{ElKhadir2020OnInequalities}. However, convex forms of degree at least four and with a sufficiently large number of variables are not necessarily sums of squares~\cite{Blekherman2012ChapterSquares}, with the first explicit example of such a form appearing in~\cite{Saunderson2023ASquares}. It remains an open problem to characterize the degrees and numbers of variables for which convex forms are sums of squares.

Previous work studying hyperbolic certificates of non-negativity has mostly focused on relating hyperbolic certificates of non-negativity to sums of squares. A hyperbolic polynomial $p$ is said to be weakly SOS-hyperbolic if every Wronskian of the form~\eqref{eq:wron} is a sum of squares. A hyperbolic polynomial is said to be SOS-hyperbolic if every B\'ezoutian of the form~\eqref{eq:bezoutian} is a matrix sum of squares. 

If $p$ is weakly SOS-hyperbolic, then any hyperwron formed from $p$ is a sum of squares. Similarly if $p$ is SOS-hyperbolic then any hyperzout formed from $p$ is a sum of squares. 

It is known (see~\cite{Saunderson2019CertifyingOptimization} and~\cite[Theorem 6.3]{Blekherman2023LinearContainment}) that there are hyperbolic polynomials of degree $d$ in $n$ variables that are not weakly SOS-hyperbolic whenever $n\geq 4$ and $d\geq 4$ or $n\geq 6$ and $d\geq 3$. Conversely if $n=3$ or $d=2$ or $(n,d) = (4,3)$, every hyperbolic polynomial of degree $d$ in $n$ variables is SOS-hyperbolic. It is not known whether hyperbolic polynomials of degree $3$ in $5$ variables are (weakly) SOS-hyperbolic. In the other direction, all sums of squares are hyperwrons (see Section~\ref{section: Relationship between hyperbolic certificates and Sum of squares certificates}).

However, in contrast to previous work, the main focus of this paper is on developing our understanding of the relationship between polynomials with hyperbolic certificates of non-negativity and all non-negative polynomials.

\subsection{Outline}

The rest of this paper is organized as follows. In Section~\ref{section: Preliminaries}, we recall some basic facts on semi-algebraic sets, hyperbolic polynomials and sums of squares. In Sections~\ref{sec: Hyperbolic-Wronskian certificates} and~\ref{sec: Hyperbolic-Bezoutian certificates}, respectively, we summarize the basic constructions of non-negative polynomials from Wronskians and B\'ezoutians of hyperbolic polynomials. In Section~\ref{section: Relationship between hyperbolic certificates and Sum of squares certificates} we establish new results on the structure of hyperbolic certificates of non-negativity for sums of squares. Section~\ref{section: Dimension Analysis} establishes Theorems~\ref{theorem: m_y no Wronskian} and~\ref{theorem: m_y no Bezoutian} showing that there exist non-negative polynomials that are not hyperwrons, and degree-restricted hyperzouts, respectively. Section~\ref{section: An example non-negative quartic homogeneous polynomial that is not hyperwron} gives an explicit example of a non-negative polynomial that is not a sum of hyperwrons. The paper concludes with a discussion of related open questions in Section~\ref{section: Discussions}. 

%% file: Main_Body/Preliminaries.tex
\section{Preliminaries}
\label{section: Preliminaries}
We begin by introducing some basic notation and terminology. Let $\langle a,b \rangle$ denote the inner product of $a$ and $b$ and $\|\cdot\|$ the $L_2$-norm. Denote by $F_{m,2y}$ the set of homogeneous polynomials with coefficients in $\R$ of degree $2y$ in $m$ real variables. Let $P_{m,2y}\subseteq F_{m,2y}$ be the convex cone of non-negative homogeneous polynomials in $m$ variables and degree $2y$, i.e., 
    \[P_{m,2y}=\left \{p \in F_{m,2y}\;:\;  p(x)\geq 0 \; \text{ for all } x\in \R^m \right \}.\]
Let $\Sigma_{m,2y}\subseteq F_{m,2y}$ denote the convex cone of  homogeneous polynomials that are sums of squares, i.e., 
\[\Sigma_{m,2y}=\left\{p\in F_{m,2y}\;:\; p=\sum_i q_i^2 \text{ for some } q_1,q_2 \ldots \in F_{m,y}\right \}.\] Since any sum of squares is non-negative, $\Sigma_{m,2y}\subseteq P_{m,2y}$.

\subsection{Semi-algebraic sets and functions}
Next, we summarize basic facts about semialgebraic sets and semialgebraic functions that we use throughout the paper.
\begin{definition}[Bochnak {\cite[Definition 2.1.4]{Bochnak1998RealGeometry}}]
\label{definition: semi-algebraic set}
A \emph{semi-algebraic} subset of $\R^n$ is a subset of the form
\[\bigcup_{i=1}^s \bigcap_{j=1}^{r_i} \left \{x\in \R^n:f_{i,j} *_{i,j} 0 \text{ for $i=1,\ldots,s$ and $j=1,\ldots, r_i$ } \right \},\]
where $f_{i,j}$ is in the polynomial ring $\R[x_1,x_2,\dots,x_n]$ and $*_{i,j}$ denotes either $<$ or $=$.
\end{definition}
\begin{definition}[Bochnak {\cite[Definition 2.2.5]{Bochnak1998RealGeometry}}]
\label{definition: semi-algebraic function}
Let $A\subseteq \R^m$ and $B\subseteq \R^n$ be semi-algebraic sets. A mapping $f:A\rightarrow B$ is \emph{semi-algebraic} if its graph, $\left\{(x,f(x))\in \R^{m+n}: x\in A\right\}$ is a semi-algebraic subset of $\R^{m+n}$.
\end{definition}
Non-negative polynomials $P_{m,2y}$, sums of squares $\Sigma_{m,2y}$, and their set difference $P_{m,2y}\setminus \Sigma_{m,2y}$, are all semialgebraic subsets of $F_{m,2y}$. We briefly establish these well-known facts, next. 
\begin{lemma}
\label{lemma: Pm2y semi-algebraic}
The set $P_{m,2y}$ is a semi-algebraic subset of $F_{m,2y}$.
\end{lemma}
\begin{proof}
Let $A_{m,2y}$ denote the collection of non-negative integer vectors of length $m$ with $L^1$-Norm of $2y$, namely \[A_{m,2y}=\left\{a\in \Z_{\geq0}^m \; : \; \sum_{i=1}^ma_i=2y \right \}.\] These are all the possible exponents of forms of degree $2y$ in $m$ variables. Let \[ Q=\left \{(a,x)\in \R^{|A_{m,2y}|}\times \R^m:\sum_{I\in A_{m,2y}} a_I x^I<0 \right\}.\] This is a semi-algebraic set by Definition~\ref{definition: semi-algebraic set}. The projection of $Q$ onto $\R^{|A_{m,2y}|}$, namely \[\Pi(Q)=\left\{a \in  \R^{|A_{m,2y}|}:\exists x \in \R^m \text{ such that } \sum_{I\in A_{m,2y}} a_I x^I<0\right\},\] is semi-algebraic~\cite[Theorem 2.2.1]{Bochnak1998RealGeometry}. Let $\Pi^c(Q)$ denote the complement of $\Pi(Q)$. Therefore, $P_{m,2y} = F_{m,2y}\setminus \Pi(Q)=F_{m,2y} \cap \Pi^c(Q)$ is semi-algebraic since the complement and finite intersection of semi-algebraic set is semi-algebraic~\cite[Section 2.1]{Bochnak1998RealGeometry}. 
\end{proof}
\begin{lemma}
\label{lemma: SOS semi-algebraic}
The set $\Sigma_{m,2y}$ is a semi-algebraic subset of $F_{m,2y}$.
\end{lemma}
\begin{proof}
$\Sigma_{m,2y}$ is a projected spectrahedron, which is always a semialgebraic set~\cite[Theorem 6.17]{2012SemidefiniteGeometry}.
\end{proof}
\begin{lemma}
\label{lemma: P-sum semi-algebraic}
The set $P_{m,2y}\setminus \Sigma_{m,2y}$ is a semi-algebraic subset of $F_{m,2y}$. 
\end{lemma}
\begin{proof}
By Lemmas \ref{lemma: Pm2y semi-algebraic} and \ref{lemma: SOS semi-algebraic}, $P_{m,2y}$ and $\Sigma_{m,2y}$ are both semi-algebraic. 

Let $\Sigma_{m,2y}^c=F_{m,2y} \setminus \Sigma_{m,2y}$ denote the complement of $\Sigma_{m,2y}$. Then $P_{m,2y}\setminus \Sigma_{m,2y}=P_{m,2y}\cap \Sigma_{m,2y}^c$. By \cite[Section 2.1]{Bochnak1998RealGeometry}, the complement of a semi-algebraic set is semi-algebraic and the intersection of two semi-algebraic sets are semi-algebraic. Therefore, $P_{m,2y}\setminus \Sigma_{m,2y}$ is semi-algebraic.
\end{proof}
There is a well-defined notion of dimension for semi-algebraic sets. See, for instance,  Bochnak \cite[Section 2.8]{Bochnak1998RealGeometry} for the precise definition. For our purposes, we only make use of certain basic properties of dimension for semi-algebraic sets.
\begin{lemma}
\label{lemma: subset-dimension}
If $A$ and $B$ are semi-algebraic subsets of $\R^n$ and $A\subseteq B$ then $\dim (A) \leq \dim (B)$.
\end{lemma}
\begin{proof}
Since $A\subseteq B$ we have that $A \cup B = B$. By \cite[Proposition 2.8.5]{Bochnak1998RealGeometry}, \[\dim(B)=\dim(A \cup B)=\max \{\dim(B),\dim(A)\} \geq \dim(A). \qedhere\]
\end{proof}
In what follows we often encounter unions of images of semialgebraic maps. When the union is finite, these are semialgebraic sets.
\begin{lemma}
\label{lemma: union of semi-algebraic map}
Let $\Omega$ be a finite set. For each $i\in \Omega$ let $a_i$ be a positive integer and let $\gamma_i:\R^{a_i}\rightarrow \R^b$ be a semi-algebraic map. Then the set $\bigcup_{i\in \Omega}\gamma_i \left(\R^{a_i} \right)$ is semi-algebraic.
\end{lemma}
\begin{proof}
Given $\gamma_i$ is a semi-algebraic map, by~\cite[Proposition 2.2.7]{Bochnak1998RealGeometry}, the set $\gamma_i \left(\R^{a_i} \right)$ is semi-algebraic. By~\cite[Section 2.1]{Bochnak1998RealGeometry}, the finite union of semi-algebraic sets is semi-algebraic. 
\end{proof}

\subsection{Hyperbolic polynomials}
\label{section: hyperbolic polynomials}
Next, we recall the definitions of hyperbolic polynomials and hyperbolicity cones, and summarize basic properties that we use throughout the paper. A homogeneous polynomial $p\in F_{n,d}$ is \emph{hyperbolic with respect to $e\in \R^n$} if
\begin{itemize}
\item $p(e)>0$ and
    \item for all $x\in \R^n$, the univariate polynomial $p(te-x)$, in the variable $t\in \R$, has only real roots.
\end{itemize}
Note that throughout the paper we use $n$ for the number of variables of a hyperbolic polynomial and $d$ for the degree of a hyperbolic polynomial, whereas for general homogeneous polynomials of even degree we use $m$ and $2y$ for the number of variables and the degree, respectively. 
Denote by $\text{Hyp}_{n,d}(e)$ the set of homogeneous polynomials of degree $d$ in $n$ variables that are hyperbolic with respect to $e$. If $p\in \text{Hyp}_{n,d}(e)$ and $x\in \R^n$, we denote the roots of $t \mapsto p(te-x)$ as $\lambda_1^{p,e}(x)\geq \lambda_2^{p,e}(x)\geq \dots \geq \lambda_{d}^{p,e}(x)$, which are also known as the \emph{hyperbolic eigenvalues} of $p$ with respect to $e$. Define the \emph{multiplicity} of $x$ to be the multiplicity of $0$ as a hyperbolic eigenvalue of $x$ with respect to $p$ and $e$. The associated \emph{hyperbolicity cone} is 
\[\Lambda_+(p,e)= \left \{x\in \R^n:\lambda_i^{p,e}(x) \geq 0 \; \text{ for all }i=1,2,\ldots,d \right\}.\] 
It turns out that any such hyperbolicity cone is a closed convex cone~\cite{Garding1959AnPolynomials}. Let $\Lambda_{++}(p,e)$ ($\partial \Lambda_+(p,e)$) denote the interior (respectively, boundary) of the hyperbolicity cone $\Lambda_+(p,e)$. 
The following result says that any direction in the interior of $\Lambda_+(p,e)$ is a direction of hyperbolicity for $p$.
\begin{proposition}[G{\aa}rding {\cite[Section 2]{Garding1959AnPolynomials}}]
\label{proposition: hyperbolicity cone inclusion}
If $p$ is hyperbolic with respect to $e$ and $c\in \Lambda_{++}(p,e)$ then $p$ is hyperbolic with respect to $c$ and $\Lambda_{++}(p,c) = \Lambda_{++}(p,e)$.

\end{proposition}

More properties regarding hyperbolic polynomials, hyperbolic eigenvalues and hyperbolicity cones can be found in, for example, \cite{Bauschke2001HyperbolicAnalysis} and \cite{Renegar2006HyperbolicRelaxations}. 

If $p\in F_{n,d}$ is a homogeneous polynomial and $e\in \R^n$ then we use the notation $D_ep$ to denote the directional derivative of $p$ in the direction $e$, i.e., $D_ep(x) = \left.\frac{d}{dt}p(x+te)\right|_{t=0}$. If $a,e\in \R^n$ then we use the notation $D_{ae}^2p(x):= D_a D_ep(x)$ for the iterated directional derivative. 

If $p$ is hyperbolic with respect to $e$, then (see, e.g.,~\cite{Garding1959AnPolynomials,Renegar2006HyperbolicRelaxations}) the directional derivative $D_ep$ is also hyperbolic with respect to $e$. Furthermore, $\Lambda_+(D_ep,e) \supseteq \Lambda_+(p,e)$, i.e., the hyperbolicity cone of the directional derivative contains the hyperbolicity cone of $p$~\cite{Renegar2006HyperbolicRelaxations}.

Next, we show that directional derivatives of hyperbolic polynomials in directions that are in the boundary of the hyperbolicity cone enjoy similar properties to directional derivatives in interior directions. 
We first summarize two technical facts about directional derivatives, and convergent sequences of real-rooted univariate polynomials.

\begin{lemma}
\label{lemma: homogeneous polynomial equivalent terms Taylor's expansion}
    Given $x,u\in \R^n$ and $p\in F_{n,d}$, it follows that $D_u^kp(x)=\frac{k!}{(d-k)!}D_x^{d-k}p(u)$.
\end{lemma}
\begin{proof}
   Let $t,\lambda \in\R$, and 
   expand $p(\lambda x + tu)$ in powers of $t$ and $\lambda$ as 
   \[p(\lambda x +tu) = \sum_{k=0}^{d} t^k\lambda^{d-k}a_k(u,x)\]
   for some polynomials $a_k$. 
   On the one hand we have that 
   \[ a_k(u,x) = \frac{1}{k{!}} \left.\frac{\partial^k}{\partial t^k} p(\lambda x + tu)\right|_{\lambda = 1,t=0} = \frac{1}{k{!}}D_u^kp(x).\]
   On the other hand, we have that 
   \[ a_k(u,x) = \frac{1}{(d-k){!}} \left.\frac{\partial^{d-k}}{\partial \lambda^{d-k}} p(\lambda x + tu)\right|_{\lambda = 0,t=1} = \frac{1}{(d-k){!}}D_x^{d-k}p(u).\]
   Equating these two expressions for $a_k(u,x)$ completes the proof. 
\end{proof}
The following result, 
Lemma~\ref{lemma: coefficient-wise limit preserve real-rootedness}, is a standard fact about real-rooted univariate polynomials. Since we could not find a proof of this result, in this form, that is easily accessible in the literature, for completeness we include a proof herein.
\begin{lemma}
    \label{lemma: coefficient-wise limit preserve real-rootedness}
    Let $(G_n)_{n\in \N}$ be a convergent sequence of real-rooted monic univariate polynomials of degree $d$ with real coefficients. Then 
    $G = \lim_{n\rightarrow \infty} G_n$ is a real-rooted monic univariate polynomial of degree $d$ with real coefficients.
\end{lemma}
\begin{proof}
\label{lemma proof: coefficient-wise limit preserve real-rootedness}

Let $(G_n)_{n\in \N}$ be a convergent sequence of real-rooted monic univariate polynomials of degree $d$ with real coefficients. Suppose $(G_n)$ converges to $G=z^d + \sum_{v=0}^{d-1}a_vz^v=\prod_{j=1}^k(z-z_j)^{m_j}$ where $m_1+m_2+\dots+m_k=d$ is a monic univariate polynomial of degree $d$ with distinct zeros $z_1,\dots,z_k$ of multiplicities $m_1,\dots,m_k$. 
Since $G_n\rightarrow G$, it follows that for every $\delta>0$, there exists some positive integer $p(\delta)$ (depending on $\delta$) such that if 
$G_{p(\delta)} = z^d +\sum_{v=0}^{d-1}b_vz^v$ then $|b_v-a_v|<\delta$ for all $v=0,1,\ldots,d-1$.

Arguing by contradiction, we  assume that $G$ has at least two roots that are not real. Let $z_c$ be one of the complex roots of multiplicity $m_c$ in the form $(y+xi)^{m_c}$, where $i$ denotes the imaginary number and $x,y\in \R$ with $x\neq 0$. 
Fix some $\varepsilon$ that satisfies $0
<\varepsilon<|x|/2$. The continuity theorem for monic univariate polynomials  (see, for example,~\cite[Theorem 1.3.1]{Rahman2002AnalyticPolynomials} or~\cite{Uherka1977OnCoefficients}) tells us that there exists $\delta>0$ 
such that whenever $F = \sum_{v=0}^{d}b_vz^v$ satisfies $|b_v-a_v|<\delta$ for $v=0,1,\ldots,d-1$, $F$ has exactly $m_c$ roots in the open disc
\[\mathcal{D}(z_c,\varepsilon)\coloneqq\{z\in \C:|z-z_c|<\varepsilon\}.\]
Note that $\mathcal{D}(z_c,\varepsilon) \cap \R=\emptyset$. Therefore, by choosing $F = G_{p(\delta)}$, we see that $G_{p(\delta)}$ has at least $m_c>0$ complex roots, which contradicts our assumption that the sequence $(G_n)_{n\in \N}$ consists of real-rooted polynomials. We can, therefore, conclude that $G$ is real-rooted.
\end{proof}

We are now in a position to prove the extended result entailing relations of hyperbolicity cones with respect to directional derivatives to the case of directions in the boundary of the cone.
\begin{proposition}
\label{prop: boundary of hyperbolicity cone directional derivative relaxation}
    Let $p\in F_{n,d}$ is hyperbolic with respect to $e\in \R^n$, $u \in \partial \Lambda_+(p,e)$ and $x\in \R^n$. Then, either $D_up(x)$ is identically zero or
    \begin{enumerate}[(i)]
        \item $D_up(x)$ is hyperbolic with respect to $e$ and
        \item $\Lambda_+(D_up(x),e) \supseteq \Lambda_+(p(x),e)$.
    \end{enumerate}
\end{proposition}
\begin{proof}
If $D_up$ is identically zero then we are done. As such, we assume that $D_up$ is not identically zero.

We start by establishing $(i)$ whenever $D_up$ is not identically zero. To prove $D_up(x)$ is hyperbolic with respect to $e$, we require (a) $D_up(e)>0$ and (b) $D_up(x+te) \in \R[t]$ is real-rooted for all $x\in \R^n$.

To establish (a), it follows from Lemma~\ref{lemma: homogeneous polynomial equivalent terms Taylor's expansion} that $D_up(e)=\frac{1}{(d-1)!}D_{e}^{d-1}p(u)\geq 0$ since $u\in \Lambda_+(p,e)\subseteq\Lambda_+(D_{e}^{d-1}p,e)$.  If $D_up(e)>0$ we are done, so assume that $D_up(e) = D_{e}^{d-1}p(u)=0$. Then it is necessarily the case that $u$ has multipicity $d$ with respect to $(p,e)$. Let $e'$ be an arbitrary point in $\Lambda_{++}(p,e)$. Since the multiplicity of $u$ with respect to $(p,e)$ is the same as the multiplicity of $u$ with respect to $(p,e')$~\cite[Proposition 22]{Renegar2006HyperbolicRelaxations}, it follows that $D_up(e') = 0$ for all $e'\in \Lambda_{++}(p,e)$. Since $D_up$ is a polynomial, it follows that $D_up$ is identically zero, contradicting our assumption.

For (b), take a sequence $u_j\in \Lambda_{++}(p,e)$ such that $u_j$ converges to some $u\in \partial\Lambda_+(p,e)$. Let $x\in \R^n$ be arbitrary. Consider the sequence $G_j(t) = \frac{D_{u_j}p(te+x)}{D_{u_j}p(e)}$. Each element in the sequence is monic and real-rooted, since $u_j\in \Lambda_{++}(p,e)$ implies that $D_{u_j}p(x)$ is hyperbolic with respect to $e$. Since, $\frac{D_{u}p(te+x)}{D_{u}p(e)} = \lim_{j\rightarrow\infty}G_j(t)$, it follows from Lemma~\ref{lemma: coefficient-wise limit preserve real-rootedness} that $\frac{D_{u}p(te+x)}{D_{u}p(e)}$ is real-rooted. Since $x$ was arbitrary, and $D_up(e)>0$, it follows that $D_up(x+te)$ is real-rooted for all $x$.

Next, we prove $(ii)$. To show $\Lambda_+(D_up,e) \supseteq \Lambda_+(p,e)$,  it is sufficient to demonstrate that $\Lambda_{++}(D_up,e)\supseteq \Lambda_{++}(p,e)$ since the result then follows by taking the closure. 
From $(i)$, we have shown that $D_up$ is hyperbolic with respect to any $e'\in \Lambda_{++}(p,e)$. Therefore $D_up(e')>0$ for all $e'\in \Lambda_{++}(p,e)$. Since the hyperbolicity cone of $D_up$ is the connected component of $\{x:D_up(x)\neq 0\}$ containing $e$~\cite[Proposition 1]{Renegar2006HyperbolicRelaxations}, it follows that $\Lambda_{++}(D_up,e) \supseteq \Lambda_{++}(p,e)$. 
\end{proof}

\subsection{Relationship between non-negative polynomials and sums of squares}
\label{sec: Relationship between nonnegative polynomials and sum of squares polynoials}
Our later results rely on Hilbert's classification of the degrees and number of variables for which non-negative homogeneous polynomials are always sums of squares. 

\begin{theorem}[Hilbert \cite{Hilbert1888UeberFormenquadraten}]
\label{theorem: Hilbert}
Let $m$ and $y$ be positive integers. Then $P_{m,2y} = \Sigma_{m,2y}$ if and only if either:
\begin{enumerate}[(i)]
    \item$m\leq 2$ (at most two variables)
    \item $m=3$ and $y=2$ (three variables and degree four)
    \item $y=1$ (quadratic forms).
\end{enumerate}
\end{theorem}

The ``only if'' direction of the theorem implies, for instance, that the cones $\Sigma_{3,6}$ and $\Sigma_{4,4}$ are strictly contained in the cones $P_{3,6}$ and $P_{4,4}$, respectively. In particular, there are non-negative polynomials that are not sums of squares.

%% file: Main_Body/Hyperbolic_Certificates.tex
\section{Hyperbolic certificates of non-negativity}

\label{section: hyperbolic non-negativity certificates}

In this section we define the families of non-negative polynomials, arising from hyperbolic polynomials, that play a central role in the paper. These come in two flavours. The first are non-negative polynomials arising as certain Wronskians of hyperbolic polynomials, which we call hyperwrons (see Section~\ref{sec: Hyperbolic-Wronskian certificates}). The second are non-negative polynomials arising from certain B\'ezoutians of hyperbolic polynomials, which we call hyperzouts (see Section~\ref{sec: Hyperbolic-Bezoutian certificates}). The family of hyperwrons is a subset of the family of hyperzouts. We refer readers to~\cite{Kummer2012HyperbolicSquares} and~\cite{Saunderson2019CertifyingOptimization} for the proof of non-negativity of these families of polynomials.

Hyperwrons are simpler to work with because, given the degree of a hyperwron, there is a finite set of possible degrees of hyperbolic polynomials that could be used to represent that hyperwron. In contrast, given a hyperzout, we are not aware of any \emph{a priori} upper bound on the degree of a hyperbolic polynomial whose B\'ezoutian gives rise to the hyperzout. To mitigate this issue, we consider a subset of hyperzouts, that we call degree-restricted hyperzouts, consisting of hyperzouts for which we explicitly constrain the degree of the hyperbolic polynomial used in their construction. We introduce these degree-restricted hyperzouts in Section~\ref{sec: Hyperbolic-Bezoutian certificates}.

\subsection{Hyperbolic-Wronskian certificates}
\label{sec: Hyperbolic-Wronskian certificates}

In this section, we define what it means for a homogeneous polynomial to have a hyperbolic-Wronskian certificate of non-negativity. We also define the set of hyperwrons. Throughout, we use the notation
$F_{m,k}^n$ to denote $n$-tuples of homogeneous polynomials of degree $k$ in $m$ variables, i.e., $F_{m,k}^n=\underbrace{F_{m,k}\times \dots \times F_{m,k}}_{\textup{$n$ copies}}$, and interpret $\phi\in F_{m,k}^n$ as a polynomial map $\phi:\R^m\rightarrow \R^n$ that is homogeneous of degree $k$.

If $p\in \textup{Hyp}_{n,d}(e)$ is a hyperbolic polynomial, and $u$ and $v$ are in the associated hyperbolicity cone $\Lambda_+(p,e)$, then the Wronskian of the univariate polynomials $p_{x,u}(t) = p(x+tu)$ and $D_vp_{x,u}(t) = D_vp(x+tu)$, i.e., 
\[ D_up(x)D_vp(x) - p(x)D_{uv}^2p(x),\]
is a non-negative homogeneous polynomial of degree $2(d-1)$~\cite[Theorem 3.1]{Kummer2012HyperbolicSquares}. Further non-negative polynomials can be generated by composing with a homogeneous polynomial map $\phi\in F_{m,k}^n$.

\begin{definition}
    \label{definition: Wronskian certificate}
A homogeneous polynomial $q\in F_{m,2y}$ has a \emph{hyperbolic-Wronskian certificate of non-negativity} if there exist positive integers $k,d,n$, a hyperbolic polynomial $p\in \textup{Hyp}_{n,d} (e)$, $u,v\in \Lambda_+(p,e)$ and map $\phi \in F_{m,k}^n$ such that 
\begin{equation*}
    \label{eq:hwc}
q(x)=D_u p(\phi(x)) D_v p(\phi(x))-p(\phi(x)) D_{uv}^2 p(\phi(x)).
\end{equation*}
\end{definition}
We say that a homogeneous polynomial $q\in F_{m,2y}$ is a \emph{hyperwron} if it has a hyperbolic-Wronskian certificate of non-negativity. We denote the collection of hyperwrons of degree $2y$ and $m$ variables by $\mathcal{W}_{m,2y}\subseteq P_{m,2y}$. In Theorem~\ref{theorem: m_y no Wronskian}, we give conditions on $m$ and $2y$ under which this containment is strict, i.e., there are non-negative polynomials that are not hyperwrons.

In the rest of this subsection, we introduce notation to help us keep track of different components of the set of hyperwrons. 
Given positive integers $m,n,d,k$ we define a map $\Theta:F_{n,d}\times \R^n\times \R^n\times F_{m,k}^n\rightarrow F_{m,2k(d-1)}$ by
\begin{equation}
\label{eq:theta-def}\Theta(p,u,v,\phi) = (D_up D_vp - p D_{uv}^2p)\circ \phi.\end{equation}
Note that $\Theta$ depends on $m,n,d,k$, but we suppress this from the notation for simplicity. If we define \[ \mathcal{S}^{n,m,d,k}_{e,W} = \{((p,u,v),\phi)\in F_{n,d}\times \R^n \times \R^n\times F_{m,k}^n\;:\;p\in \text{Hyp}_{n,d}(e),\; u,v\in \Lambda_+(p,e)\}\]
then, by definition, $\Theta(\mathcal{S}_{e,W}^{n,m,d,k}) \subseteq \mathcal{W}_{m,2k(d-1)}$. This notation is describing the hyperwrons that have hyperbolic-Wronskian certificates of non-negativity with respect to a hyperbolic polynomial of degree $d$ in $n$ variables and a map $\phi$ that is homogeneous of degree $k$. All of  $\mathcal{W}_{m,2y}$ can be built up from these pieces by varying $n$ and $(d,k)$ appropriately. 

Given a positive integer $y$, let $\Omega_y^W \coloneqq \left\{(d,k)\in\N^2:(d-1)k=y\right\}$. Note that since $y$ is positive, $(d,k)\in \Omega_{y}^{W}$ implies that $d\geq 2$ and $y\geq 1$. The set $\Omega_y^W$ describes the degrees of hyperbolic polynomials and maps $\phi$ that produce hyperwrons of degree $2y$. With this notation established, the set of hyperwrons of degree $2y$ in $m$ variables decomposes as 
\begin{equation}
\label{eq:hyperwron-basic-decomp}
    \mathcal{W}_{m,2y}= \bigcup_{(d,k)\in \Omega_y^W} \bigcup_{n\geq 1}\Theta(\mathcal{S}_{e,W}^{n,m,d,k}).
\end{equation}
The union is not disjoint---a hyperwron can have many different hyperbolic-Wronskian certificates of non-negativity. We will investigate this decomposition in more detail in Section~\ref{sec: dimension analysis --- Wronskian certificates}, as part of our 
analysis of the relationship between hyperwrons and all non-negative polynomials.

\subsection{Hyperbolic-B\'ezoutian certificates}
\label{sec: Hyperbolic-Bezoutian certificates}

In this section we discuss a generalisation of the hyperbolic-Wronskian certificate of non-negativity that is expressed in terms of the B\'ezoutian matrix of certain polynomials.

\begin{definition}[Krein and Naimark {\cite[Section 2.1]{Krein1981TheEquations}}]
\label{definition: Bezoutian}
Let $f(t),g(t)$ be univariate polynomials such that $\deg(g)\leq \deg(f)\leq d$. The B\'{e}zoutian $B_{d}(f,g)$ is the $d\times d$ matrix with $(j,l)$ entry $c_{jl}$ defined via the identity
\begin{align*}
    \frac{f(t)g(s)-f(s)g(t)}{t-s}=\sum_{j,l=0}^{d-1}c_{jl}t^{j}s^l.
    \label{eq: Bezoutian definition}
\end{align*}
\end{definition}
It will sometimes be useful to abuse notation when working with B\'{e}zoutians. In particular, if $a\in \R^{d+1}$ and $ b\in \R^{d}$, we use the notation $B_d(a,b)$ to mean $B_d(f,g)$ where $f(t) = \sum_{i=0}^{d}a_it^i$ and $g(t) = \sum_{j=0}^{d-1}b_jt^j$, identifying univariate polynomials with their coefficients in the monomial basis. We use whichever notation is more convenient, depending on the context.

If $p\in F_{n,d}$ and $u,v\in \R^n$, consider the polynomials $p_{x,u}(t)=p(x+tu)$ and $D_vp_{x,u}(t)=D_v p(x+tu)$. We think of these as univariate polynomials in $t$ (of degree at most $d$) with coefficients that are polynomials in $x$ and $u$, and linear in $v$. The parameterized B\'ezoutian
\begin{equation*}
    B_{p,u,v}(x):= B_d(p_{x,u},D_vp_{x,u})
    \label{eq:pBdef}
\end{equation*}
is a $d\times d$ matrix with entries that are polynomial in $x$ and $u$, and linear in $v$. The $(0,0)$ entry of $B_{p,u,v}(x)$ is the Wronskian of $p_{x,u}$ and $D_vp_{x,u}$, as pointed out in \cite[Remark 3.8]{Saunderson2019CertifyingOptimization} and~\cite[Remark 3.2]{Kummer2019SpectrahedralCurves}. In general, the $(j,l)$ entry (for $0\leq j,l\leq d-1$) of the parameterized B\'ezoutian $B_{p,u,v}(x)$ is homogeneous of degree $2(d-1) - (j+l)$ in $x$.

If $p\in \textup{Hyp}_{n,d}(e)$ is a hyperbolic polynomial and $u,v\in \Lambda_{+}(p,e)$, then the parameterized B\'ezoutian $B_{p,u,v}(x)$ is positive semidefinite for all $x$ 
 (see, e.g., \cite[Theorem 3.7]{Saunderson2019CertifyingOptimization} or~\cite[Theorem 2]{Kummer2017DeterminantalBezoutians}). This is, essentially, due to certain interlacing properties of $p_{x,u}$ and $D_vp_{x,u}$ (see Section~\ref{section: Discussions} for further discussion). 

To form scalar-valued homogeneous polynomials from a parameterised B\'ezoutian matrix, one can multiply on the left and right by polynomial maps of appropriate degrees. To this end, for $\mu\leq d$, let 
\begin{equation*}
    T_{\mu,k}^{m,d} = \underbrace{\{0\}\times\{0\}\times \dots \times \{0\}}_{\textup{$d-\mu$ copies}} \times F_{m,0} \times F_{m,k}\times \dots\times F_{m,(\mu-2)k}\times F_{m,(\mu-1)k}.
\end{equation*}
Then, whenever $\xi\in T_{\mu,1}^{n,d}$, the scalar-valued 
\begin{equation} 
\label{eq:bez0}\xi(x)^\intercal B_{p,u,v}(x)\xi(x)\end{equation}
is a homogeneous polynomial of degree $2(\mu-1)$.

Just as for hyperwrons, further non-negative polynomials can be generated by composing with a polynomial map $\phi\in F_{m,k}^n$ and taking $\xi\in T_{\mu,k}^{m,d}$ where $\mu \leq d$. Then, 
\begin{equation*} 
\label{eq:bez1}\xi(x)^\intercal B_{p,u,v}(\phi(x))\xi(x)
\end{equation*}
is a non-negative homogeneous polynomial of degree $2k(\mu-1)$.  As such, in this construction, hyperbolic polynomials of degree $d$ can potentially be used to generate non-negative polynomials of degree smaller than $d$. In some of our later discussion, we restrict to situations where $d<2(\mu-1)$ so that our methods give interesting results.

\begin{definition}
    \label{definition: Bezoutian certificates}
Let $q$ be a homogeneous polynomial in $m$ variables of degree $2y$. 
\begin{itemize}
    \item We say that $q$ has a \emph{hyperbolic-B\'ezoutian certificate of non-negativity} if there exist positive integers $\mu,k,n,d$ such that $\mu\leq d$ and $y=(\mu-1) k$, a hyperbolic polynomial $p\in \textup{Hyp}_{n,d} (e)$, $u,v\in \Lambda_+(p,e)$, and  maps  $\phi \in F_{m,k}^n$ and $\xi \in T_{\mu, k}^{m,d}$, such that $q(x)=\xi (x)^\intercal B_{p,u,v}(\phi(x))\xi (x)$.
    \item We say that $q$ has a \emph{degree-restricted hyperbolic-B\'ezoutian certificate of non-negativity} if, in addition, 
    either $\mu=2$ or $d \leq 2\mu-3$.
  
\end{itemize}

\end{definition}

We say that a homogeneous polynomial $q\in F_{m,2y}$ is a \emph{hyperzout} if it has hyperbolic-B\'ezoutian certificate of non-negativity. Similarly we say that $q$ is a \emph{degree-restricted hyperzout} if it has a degree-restricted hyperbolic-B\'ezoutian certificate of non-negativity. Further discussion on the motivation for the constraints on $\mu$ and $d$  imposed in the definition of degree-restricted hyperzouts is given in Remark~\ref{remark: explanation of degree restriction on degree restricted hyperzouts}.

We denote the collection of degree-restricted hyperzouts of degree $2y$ in $m$ variables by $\mathcal{B}_{m,2y}\subseteq P_{m,2y}$. In Theorem~\ref{theorem: m_y no Bezoutian} we will give conditions on $m$ and $2y$ under which there are non-negative polynomials that are not degree-restricted hyperzouts.

In the rest of this subsection, we introduce notation to keep track of the different components of the set of degree-restricted hyperzouts.
Given positive integers $m,n,d,k,\mu$ (with $\mu\leq d$) we define a map $\eta:F_{n,d}\times \R^n \times \R^n \times F_{m,k}^n\times T_{\mu,k}^{m,d} \rightarrow F_{m,2k(\mu-1)}$ by
\begin{equation}
    \label{eq:eta-def}
\eta(p,u,v,\phi,\xi)(x)=\xi(x)^\intercal B_{p,u,v}(\phi(x))\xi(x).
\end{equation} 
Note that $\eta$ depends on $m,n,d,k,\mu$, but we suppress this from the notation for simplicity. If we define 
\begin{align*}
\mathcal{S}_{e,B}^{n,m,d,k,\mu}&\!=\!\{(p,u,v,\phi,\xi)\in F_{n,d}\times \R^n\times \R^n \times F_{m,k}^n \times T_{\mu,k}^{m,d} :p\in \textup{Hyp}_{n,d}(e),\; u,v\in \Lambda_+(p,e)\}\\
& = \mathcal{S}_{e,W}^{n,m,d,k}\times T_{\mu,k}^{m,d},
\end{align*}
then, by definition, $\eta(\mathcal{S}_{e,B}^{n,m,d,k,\mu})\subseteq \mathcal{B}_{m,2k(\mu-1)}$. As with hyperwrons, all (degree-restricted) hyperzouts can be built up from these components. 

Given a positive integer $y$, let \begin{equation}
    \label{eq: OmegayB}
\Omega_{y}^B \coloneqq \left\{(d,y,2)\in\N^3:\,d\geq 2\right\} \cup \left\{(d,k,\mu)\in \N^3\;:\; \mu \leq d \leq 2\mu-3,\;k(\mu-1) = y\right\}.
\end{equation}
The set $\Omega_{y}^B$ denotes the set of degree data (for the hyperbolic polynomial, the map $\phi$ and the map $\xi$) that can produce degree-restricted hyperzouts of degree $2y$.
Using this notation, the set of degree-restricted hyperzouts of degree $2y$ and $m$ variables decomposes as  
\begin{equation}
\label{eq:hyperzout-basic-decomp}
\mathcal{B}_{m,2y}= \bigcup_{(d,k,\mu)\in \Omega_y^B}\bigcup_{n\geq 1} \eta \left (\mathcal{S}_{e,B}^{n,m,d,k,\mu} \right).
\end{equation}
This decomposition plays an important role in our analysis, in Section~\ref{sec: dimension analysis --- Degree-restricted Bezoutian certificates}, of the relationship between degree restricted hyperzouts and all non-negative homogeneous polynomials. 

The fact that the $(0,0)$ entry of $B_{p,u,v}(x)$ is the Wronskian of $p_{x,u}$ and $D_vp_{x,u}$ implies that hyperwrons are contained in the set of degree-restricted hyperzouts, i.e., $\mathcal{W}_{m,2y}\subseteq \mathcal{B}_{m,2y}\subseteq P_{m,2y}$.
This follows from the fact that 
\begin{equation}
    \label{eq:W-B-containment}
\eta\left(\mathcal{S}_{e,B}^{n,m,d,k,d}\right) \supseteq \Theta\left(\mathcal{S}_{e,W}^{n,m,d,k}\right),\end{equation}
which holds because $\eta(p,u,v,\phi,\xi) = \Theta(p,u,v,\phi)$ when $\xi(x) = (1,0,\ldots,0)\in T_{d,k}^{m,d}$.

%% file: Main_Body/Relationship_Between_Hyperbolic_And_SOS.tex
\section{Relationship between hyperwrons and sums of squares}
\label{section: Relationship between hyperbolic certificates and Sum of squares certificates}
In this section, we show that every sum of squares is a hyperwron. In~\cite[Proposition 3.13]{Saunderson2019CertifyingOptimization} it was shown that any sum of squares is a hyperzout. Proposition~\ref{proposition: SOS is subset of Wronskian} shows that if $q$ is a sum of squares, then $q$ is a hyperwron generated by a hyperbolic polynomial of degree two. As noted in Section~\ref{section: introduction}, this result essentially appears in~\cite[Remark 1]{Shu2024AlgebraicOptimization}. We include a proof for completeness and to connect with the notation used in this paper.

\begin{proposition}
\label{proposition: SOS is subset of Wronskian}
Let $q\in \Sigma_{m,2y}$ be a sum of squares. Let $n \geq n'= \binom{y+m-1}{y}$ and let $e\in \R^{n}$ be non-zero. Then, there exists a quadratic hyperbolic polynomial $p\in \textup{Hyp}_{n,2}(e)$, a polynomial map $\phi:\R^m\rightarrow \R^{n}$ that is homogeneous of degree $y$,  and elements $u,v\in \Lambda_+(p,e)$ such that \[ q(x) = D_up(\phi(x))D_vp(\phi(x)) - D_{uv}^2p(\phi(x))p(\phi(x))\] for all $x\in \R^m$. 
Equivalently, $\Sigma_{m,2y} \subseteq \Theta \left (\mathcal{S}_{e,W}^{n,m,2,y} \right)$.

\end{proposition}

\begin{proof}
Let $p(z) = \frac{1}{\|e\|^2}\langle e,z\rangle^2 - \frac{1}{2}\|z\|^2$.

We will first show that $p$ is hyperbolic with respect to $e$. 
We need to check that $p(e)>0$ and that $p(z+te)$ has $2$ real roots (counting multiplicity).

To see that $p(e)>0$ we note that 
    \begin{equation*}
        p(e)  = \frac{1}{\|e\|^2}\langle e,e\rangle^2 - \frac{1}{2}\|e\|^2
        =\|e\|^2-\frac{1}{2}\|e\|^2>0
    \end{equation*}
    since $e$ is non-zero by assumption.

To see that $p(z+te)$ has $2$ real roots, we check that the discriminant of this quadratic polynomial in $t$ is non-negative. Expanding in powers of $t$ gives
    \begin{align}
        p(z+te)&=\frac{1}{\|e\|^2}\langle e,z+te\rangle^2 - \frac{1}{2}\|z+te\|^2 \notag \\
        & = \left(\frac{1}{\|e\|^2}\langle e,z\rangle^2 - \frac{1}{2}\|z\|^2\right) + t\langle e,z\rangle + \frac{t^2}{2}\|e\|^2. \label{eq: p(z+te)}
    \end{align}
    The discriminant is 
    \begin{align*}
        &\langle e,z \rangle^2-2\|e\|^2\left(\frac{1}{\|e\|^2} \langle e,z \rangle^2-\frac{1}{2}\|z\|^2\right)\\
        &=-\langle e,z \rangle^2+\|z\|^2 \| e\|^2 \geq 0
    \end{align*}
    where we have used the Cauchy-Schwarz inequality. Therefore, all roots are real. 

Next we show that if $u=v=e$ we have that $D_up(z)D_vp(z) - D_{uv}^2p(z)p(z) = \frac{1}{2}\|e\|^2\|z\|^2$, so that the Wronskian is a sum of squares. From the expression for $p(z+te)$ in~\eqref{eq: p(z+te)} we see that 
\begin{align*}
D_ep(z) & = \left.\frac{d}{dt}p(z+te)\right|_{t=0} = \langle e,z\rangle\\
D_{ee}^2p(z) & = \left.\frac{d^2}{dt^2}p(z+te)\right|_{t=0} = \|e\|^2.
\end{align*}
A direct computation of the Wronskian gives
\begin{align}
    D_ep(z)^2-p(z)D_{ee}^2p(z) &=\langle e,z\rangle^2 - \left(\frac{1}{\|e\|^2}\langle e,z\rangle^2 - \frac{1}{2}\|z\|^2\right)\|e\|^2 \notag \\
    & = \frac{1}{2}\|e\|^2\|z\|^2\label{eq:sos-wron}.
\end{align}
Finally, since $q(x)$ is a sum of squares, we know that it has a sum of squares decomposition involving at most $\dim(F_{m,y}) = n'$ terms~\cite[Proposition 3.2]{Naldi2014NonnegativeNumber}. Therefore there exist $q_i\in F_{m,y}$ (for $i=1,2,\ldots,n'$) such that $q(x) = \sum_{i=1}^{n'}q_i(x)^2$. For $i=n'+1,\ldots,n$ let $q_i = 0\in F_{m,y}$. Let $\phi:\R^m\rightarrow \R^{n}$ be defined by
\[ \phi(x) = \frac{\sqrt{2}}{\|e\|}\begin{bmatrix}q_1(x)\\q_2(x)\\\vdots\\ q_{n}(x) \end{bmatrix}.\]
Then, from~\eqref{eq:sos-wron},
\[ D_up(\phi(x))D_vp(\phi(x)) - D_{uv}^2p(\phi(x))p(\phi(x)) = \sum_{i=1}^{n'}q_i(x)^2 = q(x),\]
completing the argument.
This shows $\Sigma_{m,2y} \subseteq \Theta \left(\mathcal{S}_{e,W}^{n,m,2,y} \right)$ whenever $n\geq n'$.

\end{proof}
It will be useful, in our later analysis, to understand families of hyperwrons and hyperzouts that are always sums of squares. We first include a useful fact about $2\times 2$ polynomial matrices.

\begin{lemma}
\label{lemma: degree 2 Bezoutian SOS decomposition}
 Let $p_2\in F_{n,2}$, $p_1\in F_{n,1}$, $p_0\in \R$ be such that $\left(\begin{smallmatrix}
        p_2(x) & p_1(x)\\  p_1(x) & p_0
    \end{smallmatrix}\right) \succeq 0$ for all $x\in \R^n$. Then there exists a $2 \times (n+1)$ matrix $M$ with polynomial entries such that \[M(x)M(x)^\intercal=\begin{pmatrix}
        p_2(x) & p_1(x)\\  p_1(x) & p_0
    \end{pmatrix}.\]
\end{lemma}
\begin{proof}
First assume that $p_0>0$. We write 
    \begin{equation}
    \label{eq:m2factorization}
        \begin{pmatrix}
        p_2(x) & p_1(x)\\  p_1(x) & p_0
    \end{pmatrix}=\begin{pmatrix}
        1 & \frac{p_1(x)}{p_0}\\ 0 &1
    \end{pmatrix}\begin{pmatrix}
        p_2(x)-\frac{p_1^2(x)}{p_0}& 0\\ 0 &p_0
    \end{pmatrix}\begin{pmatrix}
        1 & \frac{p_1(x)}{p_0}\\ 0 &1
    \end{pmatrix}^\intercal.
    \end{equation}

Since $\left(\begin{smallmatrix} 1 & \frac{p_1(x)}{p_0}\\0&1\end{smallmatrix}\right)$ is invertible and the left hand side of~\eqref{eq:m2factorization} is positive semidefinite for all $x\in \R^n$,
it follows that $p_2(x)-p_1(x)^2/p_0 \geq 0$ for all $x\in \R^n$.
    As $p_2(x)-\frac{p_1^2(x)}{p_0}\in P_{n,2}$, it must be a sum of squares. Therefore, there exist $q_1,\ldots,q_n\in F_{n,1}$ such that $p_2(x)-\frac{p_1^2(x)}{p_0}=\Sigma_{i=1}^n q_i^2(x)$ for all $x$. The required matrix $M$ is then 
    \begin{equation*}
        \label{eq:Meqpositive}
    M(x)=\begin{pmatrix} 1 & \frac{p_1(x)}{p_0}\\0 & 1\end{pmatrix}\begin{pmatrix}
        q_1(x)&q_2(x)&\cdots&q_n(x)&0\\
        0&0&\cdots&0&\sqrt{p_0}
    \end{pmatrix} = \begin{pmatrix}
        q_1(x)&q_2(x)&\cdots&q_n(x)&\frac{p_1(x)}{\sqrt{p_0}}\\
        0&0&\cdots&0&\sqrt{p_0}
    \end{pmatrix}.
    \end{equation*}
    In the case where $p_0 = 0$, it must also be the case that $p_1(x) = 0$ for all $x$. Since $p_2\in P_{n,2}$ there exist $\tilde{q}_i\in F_{n,1}$ such that $p_2(x) = \sum_{i=1}^{n}\tilde{q}_i(x)^2$. Then we can simply take $M(x) = \left(\begin{smallmatrix} \tilde{q}_1(x) & \cdots& \tilde{q}_n(x)&0\\0 & \cdots & 0&0\end{smallmatrix}\right)$.
\end{proof}
Now, we summarize  relationships between sums of squares, certain hyperwrons, and certain hyperzouts. Note that the result of~\cite[Remark 1]{Shu2024AlgebraicOptimization} follows directly from Lemma~\ref{lem:sos-hyperwon-hyperzout}. 
\begin{lemma}
    \label{lem:sos-hyperwon-hyperzout}
    Let $m$, $n$ and $y$ be positive integers and let $e\in \R^n$ be non-zero. Then 
    \[ \Sigma_{m,2y} \supseteq \bigcup_{d\geq 2}\eta\left(\mathcal{S}_{e,B}^{n,m,d,y,2}\right) \supseteq \eta\left(\mathcal{S}_{e,B}^{n,m,2,y,2}\right) \supseteq \Theta\left(\mathcal{S}_{e,W}^{n,m,2,y}\right).\]
    Moreover, if $n \geq \binom{m-1+y}{y}$ then $\Theta\left(\mathcal{S}_{e,W}^{n,m,2,y}\right) = \Sigma_{m,2y}$.
\end{lemma}
\begin{proof}
    The right-most inclusion follows from~\eqref{eq:W-B-containment} with $d=2$. The middle inclusion is obvious from the definition of the union. The fact that $\Theta\left(\mathcal{S}_{e,W}^{n,m,2,y}\right)\supseteq \Sigma_{m,2y}$ holds when $n\geq \binom{m-1+y}{y}$ follows from Proposition~\ref{proposition: SOS is subset of Wronskian}.

    It remains to show that $\eta\left(\mathcal{S}_{e,B}^{n,m,d,y,2}\right)\subseteq \Sigma_{m,2y}$ whenever $d\geq 2$. To see why this is true, we note that any element $q\in \eta\left(\mathcal{S}_{e,B}^{n,m,d,y,2}\right)$ can be written in the form
    \[ q(x) = \begin{matrix}\begin{pmatrix} \xi_0 & \xi_y(x)\end{pmatrix}\\\phantom{a}\end{matrix}\begin{pmatrix} p_2(\phi(x)) & p_1(\phi(x))\\p_1(\phi(x)) & p_0\end{pmatrix}\begin{pmatrix} \xi_0 \\ \xi_y(x)\end{pmatrix}\;\;\textup{for all $x\in \R^m$}\]
    where $p_j\in F_{n,j}$ (for $j=0,1,2$), $\xi_j\in F_{m,j}$ (for $j=0,y$), and $\left(\begin{smallmatrix} p_2(z) & p_1(z)\\p_1(z) & p_0\end{smallmatrix}\right) \succeq 0$ for all $z\in \R^n$. This is because the bottom-right $2\times 2$ principal submatrix of a B\'ezoutian of the form $B_{p,u,v}(z)$ appearing in~\eqref{eq:bez0} always has entries that are homogeneous of degrees $2$, $1$, and $0$ respectively in $z$. But, by Lemma~\ref{lemma: degree 2 Bezoutian SOS decomposition}, there exists a matrix $M(z)$ with polynomial entries, such that $\left(\begin{smallmatrix}p_2(z) & p_1(z)\\p_1(z) & p_0\end{smallmatrix}\right) = M(z)M(z)^\intercal$. It follows that 
    \[q(x) = \left\|M(\phi(x))^\intercal \begin{pmatrix}\xi_0\\\xi_y(x)\end{pmatrix}\right\|^2\]
    is a sum of squares. 
\end{proof}

The fact that the set of sums of squares coincides with the hyperwrons (and also the hyperzouts) generated by hyperbolic polynomials of degree two will play an important role in our analysis in Sections~\ref{sec: dimension analysis --- Wronskian certificates} and~\ref{sec: dimension analysis --- Degree-restricted Bezoutian certificates}. Indeed, this observation will eventually allow us to focus on polynomials that are not sums of squares, and reduce to reasoning about hyperwrons and hyperzouts that are generated by hyperbolic polynomials of degree strictly greater than two.

%% file: Main_Body/Dimension_Analysis.tex
\section{Dimension analysis}
\label{section: Dimension Analysis}

One way we might hope to show that there are non-negative polynomials that are not hyperwrons (or, indeed hyperzouts), is by comparing some notion of the size of the set of non-negative polynomials and the set of hyperwrons. Since $\mathcal{W}_{m,2y}\supseteq \Sigma_{m,2y}$, in the cases where $P_{m,2y} = \Sigma_{m,2y}$, we know that $\mathcal{W}_{m,2y} = P_{m,2y}$. Therefore, we will only find non-negative polynomials that are not hyperwrons in cases where there are non-negative polynomials that are not sums of squares. Since both $\mathcal{W}_{m,2y}$ and $P_{m,2y}$ contain $\Sigma_{m,2y}$, it follows that both sets have non-empty interior. Therefore, dimension, alone, cannot distinguish between non-negative polynomials and hyperwrons.

We could proceed by trying  to compare the semi-algebraic dimension (as defined in~\cite[Section 2.8]{Bochnak1998RealGeometry}) of the semi-algebraic set $P_{m,2y}\setminus \Sigma_{m,2y}$ with an appropriate notion of dimension for the set of all hyperwrons that are not sums of squares. However, it is not clear whether the latter set is even semi-algebraic since $\mathcal{W}_{m,2y}$ is described as an infinite union. 

Instead, to enable us to use the tools of semi-algebraic geometry, we will construct (in the proof of Theorem~\ref{theorem: Dimension bound of Wronskian from hyperbolic polynomial}) a semi-algebraic set $\Gamma_{\mathcal{W}_{m,2y}}$ that contains $\mathcal{W}_{m,2y}\setminus \Sigma_{m,2y}$, permitting a straightforward bound of the dimension of $\Gamma_{\mathcal{W}_{m,2y}}$. Our construction of $\Gamma_{\mathcal{W}_{m,2y}}$ takes the form
\begin{equation}
\label{eq:gamma-structure}
\Gamma = \bigcup_{i\in \Omega} \gamma_i(\R^{b_i}),
\end{equation} 
where $\Omega$ is a finite set, $(b_i)_{i\in \Omega}$ and $c$ are positive integers, and $\gamma_i:\R^{b_i}\rightarrow \R^c$ (for $i\in \Omega$) are semi-algebraic maps. This structure makes the dimension of $\Gamma$ straightforward to bound (see Proposition~\ref{proposition: dimension inequality non-inclusion}), and arises naturally from the decomposition of hyperwrons given in~\eqref{eq:hyperwron-basic-decomp}. We then establish conditions on $(m,2y)$ such that the dimension of $P_{m,2y}\setminus \Sigma_{m,2y}$ is strictly larger than the dimension of $\Gamma_{\mathcal{W}_{m,2y}}$, which in turn implies the existence of a non-negative polynomial that is not a hyperwron.

There are two main ideas behind the construction of the set $\Gamma_{\mathcal{W}_{m,2y}}$. The first is that we can obtain all hyperwrons that are not sums of squares by considering 
polynomials of the form $D_up(\phi(x))D_vp(\phi(x)) - p(\phi(x))D_{uv}^2p(\phi(x))$ where $p$ is hyperbolic in $n$ variables of degree $d\geq 3$, $\phi:\R^m\rightarrow \R^n$ is a homogeneous polynomial map of degree $k$, and $u,v\in \Lambda_+(p,e)$. In particular, we can exclude hyperwrons generated by hyperbolic polynomials of degree two, since these all give rise to sums of squares. The second key idea is based on the simple observation that  $D_up(\phi(x))D_vp(\phi(x)) - p(\phi(x))D_{uv}^2p(\phi(x))$ has the form $p_1(x)p_2(x)-p_3(x)p_4(x)$ where $p_1,p_2\in F_{m,(d-1)k}$, $p_3\in F_{m,dk}$ and $p_4\in F_{m,(d-2)k}$. Instead of trying to bound the dimension of hyperwrons directly, we can instead bound the dimension of expressions of the form $p_1p_2-p_3p_4$, where $p_1,p_2\in F_{m,(d-1)k}$, $p_3\in F_{m,dk}$ and $p_4\in F_{m,(d-2)k}$. In particular, the $p_i$ are polynomials in $m$ variables, even though $p$ has $n$ variables. This allows us to obtain bounds that are independent of $n$. 

We take a similar approach to understand cases in which there are non-negative polynomials that are not degree-restricted hyperzouts. We construct a semi-algebraic set $\Gamma_{\mathcal{B}_{m,2y}}$ that contains $\mathcal{B}_{m,2y}\setminus \Sigma_{m,2y}$ and that is a finite union of images of semi-algebraic maps. We then establish conditions on $(m,2y)$ such that the dimension of $P_{m,2y}\setminus \Sigma_{m,2y}$ is strictly larger than the dimension of $\Gamma_{\mathcal{B}_{m,2y}}$.

In Section~\ref{sec:dimension-basics}, we establish some basic facts about the dimension of semi-algebraic sets arising in our later arguments.
In Section~\ref{sec: dimension analysis --- Wronskian certificates} we focus on the construction of the set $\Gamma_{\mathcal{W}_{m,2y}}$ for the hyperwon case. In Section~\ref{sec: dimension analysis --- Degree-restricted Bezoutian certificates} we focus on the construction of the set
$\Gamma_{\mathcal{B}_{m,2y}}$ for the degree-restricted hyperzout case. In Section~\ref{sec:main-results} we establish sufficient conditions on $(m,2y)$, under which there are non-negative polynomials that are not hyperwons (respectively, degree-restricted hyperzouts).

\subsection{Preliminary facts about dimension of semi-algebraic sets}
\label{sec:dimension-basics}
The following simple observation bounds the dimension of semi-algebraic sets contained in a set with the same structural  form as $\Gamma_{\mathcal{W}_{m,2y}}$ or $\Gamma_{\mathcal{B}_{m,2y}}$ (defined in Sections~\ref{sec: dimension analysis --- Wronskian certificates} and~\ref{sec: dimension analysis --- Degree-restricted Bezoutian certificates}, respectively).
\begin{proposition}
\label{proposition: dimension inequality non-inclusion}
Let $\Omega$ be a finite set and let $c$ be a positive integer, and let $C\subseteq \R^c$ be a semi-algebraic set. For each $i\in \Omega$, let $b_i$ be a positive integer, let $\gamma_i:\R^{b_i} \rightarrow \R^c$ be a semi-algebraic map, and let $B_i \subseteq \R^{b_i}$ be an arbitrary set. If $C  \subseteq \bigcup_{i\in \Omega} \gamma_i(B_i)$, then $\dim(C)\leq \max_{i\in \Omega} b_i$.

\end{proposition}
\begin{proof}
Since $C\subseteq \bigcup_{\substack{i\in \Omega}} \gamma_i(B_i)$ and $B_i\subseteq \R^{b_i}$ for all $i\in \Omega$, it follows that 
\begin{equation*}
    C\subseteq \bigcup_{\substack{i\in \Omega}} \gamma_i \left(\R^{b_{i}} \right).  
    \label{eq: assumption in proposition dimension inequality non-inclusion proof}
\end{equation*}
Since $\gamma_i\left(\R^{b_{i}} \right)$ is semi-algebraic and the finite union of semi-algebraic sets is semi-algebraic (Lemma~\ref{lemma: union of semi-algebraic map}), $\bigcup_{i\in\substack{\Omega}} \gamma_i\left(\R^{b_{i}} \right)$ is semi-algebraic. By Lemma~\ref{lemma: subset-dimension} and the fact that dimension of a finite union of semi-algebraic sets is the maximum of the dimensions of the constituent sets~\cite[Proposition 2.8.5]{Bochnak1998RealGeometry}, we have \[\dim (C) \leq \dim\left(\bigcup_{\substack{i\in \Omega}} \gamma_i\left(\R^{b_{i}} \right)\right)= \max_{\substack{i\in \Omega}}\dim  \gamma_i\left(\R^{b_{i}} \right).\]

The proposition follows from the fact that $\dim  \gamma_i\left(\R^{b_{i}} \right) \leq  \dim \R^{b_{i}}=b_i$,  \cite[Theorem 2.8.8, Proposition 2.8.4]{Bochnak1998RealGeometry}. 
\end{proof}

The other set that plays a key role in our dimension-based argument is $P_{m,2y}\setminus\Sigma_{m,2y}$, the set of non-negative homogeneous polynomials that are not sums of squares. 
Lemmas~\ref{lemma: Pm2y not SOS non-empty} and~\ref{lemma: dimension equality} together show that if there is a non-negative polynomial that is not a sum of squares, then $P_{m,2y}\setminus\Sigma_{m,2y}$ is full-dimensional in all homogeneous polynomials of degree $2y$ in $m$ variables. Although this is a well-known fact, we include a proof for completeness.

\begin{lemma}
\label{lemma: Pm2y not SOS non-empty}
If $P_{m,2y}\setminus \Sigma_{m,2y}$ is non-empty, it has a non-empty interior.
\end{lemma}
\begin{proof}
Let $\hat{q}\in P_{m,2y}\setminus \Sigma_{m,2y}$, then $\hat{q}\in \Sigma_{m,2y}^c$, where $\Sigma_{m,2y}^c$ denotes the complement of the set $\Sigma_{m,2y}$ in $F_{m,2y}$. Since $\Sigma_{m,2y}$ is closed, $\Sigma_{m,2y}^c$ is open. As a result, there exists $\varepsilon >0$ such that $B(\hat{q};\varepsilon)\subseteq \Sigma_{m,2y}^c$, where $B(\hat{q};\varepsilon)$ is the open ball
\[B(\hat{q};\varepsilon)=\left\{q\in F_{m,2y}: \max_{\substack{x\in S^{n-1}}} |q(x)-\hat{q}(x)|<\varepsilon \right\},\]
where $S^{n-1}$ is the unit sphere in $\mathbb{R}^n$.

Consider $q(x)=\hat{q}(x)+\frac{\varepsilon}{2}(x_1^2+\dots+x_n^2)^y$. Observe that $q\in B(\hat{q};\varepsilon)\subseteq \Sigma_{m,2y}^c$. Also, since $q(x)\geq \frac{\varepsilon}{2} >0$ for all $x\in S^{n-1}$ it follows that $q\in \text{int } \left(P_{m,2y} \right)$. This implies 
\begin{equation*}
    q\in  \text{int}(P_{m,2y}) \cap \Sigma_{m,2y}^c
    \subseteq \text{int} \left(P_{m,2y}\setminus \Sigma_{m,2y} \right).
\end{equation*}
Here the inclusion holds because $\textup{int}(P_{m,2y})\cap \Sigma_{m,2y}^c$ is an open set contained in $P_{m,2y}\setminus \Sigma_{m,2y}$. 
This shows the interior of $P_{m,2y}\setminus \Sigma_{m,2y}$ is non-empty.
\end{proof}

\begin{lemma}
\label{lemma: dimension equality}
If $m >2,2y >2$ and $(m,2y)\neq (3,4)$ then 
\[\dim P_{m,2y}=\dim (P_{m,2y}\setminus \Sigma_{m,2y})=\dim F_{m,2y}=\binom{m+2y-1}{2y}.\]
\end{lemma}
\begin{proof}

Since $m >2,2y >2$ and $(m,2y)\neq (3,4)$, the set $P_{m,2y}\setminus \Sigma_{m,2y}$ is non-empty by Theorem~\ref{theorem: Hilbert}. By Lemma~\ref{lemma: Pm2y not SOS non-empty}, $P_{m,2y}\setminus \Sigma_{m,2y}$ has a non-empty interior. 

By~\cite[Proposition 2.2.2]{Bochnak1998RealGeometry}, if $S$ is semi-algebraic, so is the interior of $S$, $\text{int}\,S$. Denote $\R^c$ as the ambient space of $P_{m,2y}$. By~\cite[Proposition 2.8.4]{Bochnak1998RealGeometry}, a non-empty open semi-algebraic subset $U$ of $\R^n$ has $\dim(U)=n$. Combining it with Lemma~\ref{lemma: subset-dimension}, we have
\[\dim \R^c =\dim (\text{int} (P_{m,2y}\setminus \Sigma_{m,2y}))\leq \dim (P_{m,2y}\setminus \Sigma_{m,2y} )\leq \dim P_{m,2y} \leq \dim \R^c .\] We deduce from here that $\dim \R^c =\dim (P_{m,2y}\setminus \Sigma_{m,2y}) = \dim P_{m,2y}$. Since $P_{m,2y}$ is full dimensional in $F_{m,2y}$~\cite[Exercise 4.2]{2012SemidefiniteGeometry}, it follows that $\dim P_{m,2y}=\dim (P_{m,2y}\setminus \Sigma_{m,2y})=\dim F_{m,2y}$.
\end{proof}

\subsection{Wronskian certificates}
\label{sec: dimension analysis --- Wronskian certificates}
In this section, we construct a semi-algebraic set $\Gamma_{\mathcal{W}_{m,2y}}$ of the form~\eqref{eq:gamma-structure} that contains all hyperwrons which are not sums of squares. This leads to a sufficient condition for the existence of a non-negative homogeneous polynomial that is not a hyperwron.

We begin by defining a subset of hyperwrons that contains all hyperwrons which are not sums of squares. Recall from~\eqref{eq:hyperwron-basic-decomp} that the set of hyperwrons decomposes as
\[ \mathcal{W}_{m,2y}= \bigcup_{(d,k)\in \Omega_y^W} \bigcup_{n\geq 1}\Theta(\mathcal{S}_{e,W}^{n,m,d,k})\]
where $\Omega_{y}^{W} = \{(d,k)\in \mathbb{N}^2\;:\;(d-1)k = y\}$ and $\Theta$ is defined in~\eqref{eq:theta-def}.
Let
\begin{equation}
    \label{eq: Tilde OmegayW}
    \tilde{\Omega}_{y}^W = \{(d,k)\in {\Omega}_y^W\;:\; d\neq 2\}
\end{equation}
and define
\begin{equation}
\tilde{\mathcal{W}}_{m,2y}=\bigcup_{(d,k)\in \tilde{\Omega}_y^W}\bigcup_{n\geq 1}  \Theta \left(\mathcal{S}_{e,W}^{n,m,d,k} \right).
\label{eq: Tilde Wronskian collection definition}    
\end{equation}
This is the set of hyperwrons generated by hyperbolic polynomials of degree strictly greater than two. By Proposition~\ref{proposition: SOS is subset of Wronskian}, the set $\tilde{\mathcal{W}}_{m,2y}$ contains all hyperwrons that are not sums of squares. 

\begin{lemma}
\label{lem:Wtilde-property}
    Let $\tilde{\mathcal{W}}_{m,2y}$ be defined in~\eqref{eq: Tilde Wronskian collection definition}. Then, $\mathcal{W}_{m,2y}\setminus \Sigma_{m,2y}\subseteq \tilde{\mathcal{W}}_{m,2y}$.
\end{lemma}
\begin{proof}
    Let $q\in \mathcal{W}_{m,2y}\setminus \Sigma_{m,2y}$. Then 
    \[q\in \mathcal{W}_{m,2y} = \tilde{\mathcal{W}}_{m,2y} \cup \left( \bigcup_{n\geq 1} \Theta\left(\mathcal{S}_{e,W}^{n,m,2,y}\right)\right)\subseteq \tilde{\mathcal{W}}_{m,2y} \cup \Sigma_{m,2y},\]
    where the inclusion holds because any hyperwron generated from a hyperbolic polynomial of degree two is a sum of squares by Lemma~\ref{lemma: degree 2 Bezoutian SOS decomposition}. Since $q\notin \Sigma_{m,2y}$ by assumption, it follows that $q\in \tilde{\mathcal{W}}_{m,2y}$.
\end{proof}
Our aim, now, is to construct a set $\Gamma_{\mathcal{W}_{m,2y}}$ that is a finite union of images of semi-algegraic maps, (i.e., of the form~\eqref{eq:gamma-structure}) that contains $\tilde{\mathcal{W}}_{m,2y}$. To do this we use the simple, but crucial, observation that the map $\Theta$, defined in~\eqref{eq:theta-def} factors through a (low-dimensional) space, the dimension of which is \emph{independent of $n$}. 

\begin{lemma}
    \label{lem:sa-cover-W}
    Let $m$ and $y$ be positive integers and let $(d,k)\in \Omega_{y}^W$. Then 
    \[ \Theta\left(\mathcal{S}_{e,W}^{n,m,d,k}\right) \subseteq \Theta_1 \left(F_{m,(d-1)k}\times F_{m,(d-1)k}\times F_{m,dk}\times F_{m,(d-2)k}\right)\]
    where $\Theta_1:F_{m,(d-1)k}\times F_{m,(d-1)k}\times F_{m,dk}\times F_{m,(d-2)k}\rightarrow F_{m,2y}$, defined by 
    $\Theta_1(p_1,p_2,p_3,p_4) = p_1p_2 - p_3p_4$, is a semi-algebraic map.
\end{lemma}
\begin{proof}
    The map $\Theta$ factors as $\Theta = \Theta_1\circ \Theta_2$ where
    $\Theta_2:F_{n,d}\times \R^n\times \R^n\times F_{m,k}^n \rightarrow F_{m,(d-1)k}\times F_{m,(d-1)k}\times F_{m,dk}\times F_{m,(d-2)k}$ is defined by
    \[ \Theta_2(p,u,v,\phi) = (D_up \circ \phi, D_vp \circ \phi, p \circ \phi, D^2_{uv}p \circ \phi).\]
    Since $\Theta_2\left(\mathcal{S}_{e,W}^{n,m,d,k}\right)\subseteq F_{m,(d-1)k}\times F_{m,(d-1)k}\times F_{m,dk}\times F_{m,(d-2)k}$ it follows directly that 
\[ \Theta\left(\mathcal{S}_{e,W}^{n,m,d,k}\right) \subseteq \Theta_1\left(F_{m,(d-1)k}\times F_{m,(d-1)k}\times F_{m,dk}\times F_{m,(d-2)k}\right).\]
    To see that $\Theta_1$ is a semi-algebraic map, we note that its graph is 
    \[ \{(p_1,p_2,p_3,p_4,q)\in F_{m,(d-1)k}\times F_{m,(d-1)k}\times F_{m,dk}\times F_{m,(d-2)k} \times F_{m,2y}\;:\; q = p_1p_2 - p_3p_4\}.\]
    This is a semi-algebraic set since it is defined by the common solution of $\dim(F_{m,2y}) = \binom{2y+m-1}{2y}$ quadratic equations, obtained by equating the coefficients on the left and right hand sides of the polynomial identity $q = p_1p_2 - p_3p_4$. 
\end{proof}

We are now in a position to give a sufficient condition that implies the existence of a non-negative polynomial that is not a hyperwron.
\begin{theorem}
\label{theorem: Dimension bound of Wronskian from hyperbolic polynomial}
Suppose $m>2$, $2y>2$ and $(m,2y)\neq (3,4)$. There exists a non-negative homogeneous polynomial of degree $2y$ in $m$ variables that is not a hyperwron whenever
\begin{align}
        &\dim{P_{m,2y}}=\binom{2y+m-1}{2y} \notag \\ 
        &>\max_{(d,k)\in \Tilde{\Omega}_{y}^W} 2\binom{m+(d-1)k-1}{(d-1)k}+\binom{m+dk-1}{dk}+\binom{m+(d-2)k-1}{(d-2)k},  \label{eq: hyperbolic-Wronskian inequality}
\end{align}
where  
$\Tilde{\Omega}_{y}^W$ is defined in~\eqref{eq: Tilde OmegayW}. 
\end{theorem}

\begin{proof}
Let $\Gamma_{\mathcal{W}_{m,2y}}\subseteq F_{m,2y}$ be defined by
\begin{equation*}
    \label{eq:gamW-def}
    \Gamma_{\mathcal{W}_{m,2y}} = \bigcup_{(d,k)\in \tilde{\Omega}_y^W} \Theta_1\left(F_{m,(d-1)k}\times F_{m,(d-1)k}\times F_{m,dk}\times F_{m,(d-2)k}\right),
\end{equation*}
where $\Theta_1$ is defined in Lemma~\ref{lem:sa-cover-W}.
We first claim that $\Gamma_{\mathcal{W}_{m,2y}} \supseteq \tilde{\mathcal{W}}_{m,2y} \supseteq \mathcal{W}_{m,2y}\setminus \Sigma_{m,2y}$. 
This holds because 
\begin{align*} \Gamma_{\mathcal{W}_{m,2y}} & = \bigcup_{(d,k)\in \tilde{\Omega}_y^W} \Theta_1\left(F_{m,(d-1)k}\times F_{m,(d-1)k}\times F_{m,dk}\times F_{m,(d-2)k}\right)\\
& = \bigcup_{(d,k)\in \tilde{\Omega}_y^W}\bigcup_{n\geq 1} \Theta_1\left(F_{m,(d-1)k}\times F_{m,(d-1)k}\times F_{m,dk}\times F_{m,(d-2)k}\right)\\
& \supseteq \bigcup_{(d,k)\in \tilde{\Omega}_y^W}\bigcup_{n\geq 1} \Theta\left(\mathcal{S}_{e,W}^{n,m,d,k}\right)\\
& = \tilde{\mathcal{W}}_{m,2y} \supseteq \mathcal{W}_{m,2y}\setminus \Sigma_{m,2y},
\end{align*}
where the first equality is the definition of $\Gamma_{\mathcal{W}_{m,2y}}$, the second equality holds because $\Theta_1\left(F_{m,(d-1)k}\times F_{m,(d-1)k}\times F_{m,dk}\times F_{m,(d-2)k}\right)$ is independent of $n$, the first containment follows from Lemma~\ref{lem:sa-cover-W}, and the final containment follows from Lemma~\ref{lem:Wtilde-property}.

To complete the proof, we assume that $\mathcal{W}_{m,2y} = P_{m,2y}$ and derive a contradiction.
Lemma~\ref{lem:sa-cover-W} shows that $\Theta_1$ is a semi-algebraic map. The set $\Tilde{\Omega}_y^W$ is finite by construction. Therefore, Proposition~\ref{proposition: dimension inequality non-inclusion} and the definition of $\Gamma_{\mathcal{W}_{m,2y}}$ imply that 
\begin{align} 
\dim(\Gamma_{\mathcal{W}_{m,2y}}) &\leq \max_{(d,k)\in \tilde{\Omega}_{y}^W} \dim(F_{m,(d-1)k}\times F_{m,(d-1)k}\times F_{m,dk}\times F_{m,(d-2)k})\nonumber\\
&= \max_{(d,k)\in \tilde{\Omega}_{y}^W}2\binom{m+(d-1)k-1}{(d-1)k} + \binom{m+dk-1}{dk} + \binom{m+(d-2)k-1}{(d-2)k}\nonumber\\
& < \binom{m+2y-1}{2y}\label{eq:dim-ineq1},
\end{align}
where the last inequality follows from the assumption that the inequality~\eqref{eq: hyperbolic-Wronskian inequality} holds.

Observe that if $\mathcal{W}_{m,2y} = P_{m,2y}$ then $\mathcal{W}_{m,2y}\setminus \Sigma_{m,2y} = P_{m,2y}\setminus \Sigma_{m,2y}$. On the other hand, we have established 
\[\Gamma_{\mathcal{W}_{m,2y}} \supseteq \mathcal{W}_{m,2y}\setminus \Sigma_{m,2y} = P_{m,2y}\setminus \Sigma_{m,2y}.\]
This, together with Lemmas~\ref{lemma: subset-dimension} and~\ref{lemma: dimension equality} (and the assumption that $m>2$, $2y>2$ and $(m,2y)\neq (3,4)$) implies the inequality 
\[
 \dim(\Gamma_{\mathcal{W}_{m,2y}}) \geq \dim(P_{m,2y}\setminus \Sigma_{m,2y}) = \binom{m+2y-1}{2y},\]
contradicting~\eqref{eq:dim-ineq1}. Therefore there must exist a non-negative polynomial that is not a hyperwron.
\end{proof}

\subsection{Degree-restricted-Bezoutian certificates}
\label{sec: dimension analysis --- Degree-restricted Bezoutian certificates}
In this section, we construct a semi-algebraic set $\Gamma_{\mathcal{B}_{m,2y}}$ of the form~\eqref{eq:gamma-structure} that contains all degree-restricted hyperzouts that are not sums of squares. This leads to a sufficient condition for the existence of a non-negative homogeneous polynomial that is not a degree-restricted hyperzout. 
Our arguments in this section follow the same strategy employed in Section~\ref{sec: dimension analysis --- Wronskian certificates} for the case of hyperwrons.

We begin by defining a subset of degree-restricted hyperzouts which contains all degree-restricted hyperzouts that are not sums of squares.
Recall from~\eqref{eq:hyperzout-basic-decomp} that the set of degree-restricted hyperzouts decomposes as
\[ \mathcal{B}_{m,2y}= \bigcup_{(d,k,\mu)\in \Omega_y^B} \bigcup_{n\geq 1}\eta(\mathcal{S}_{e,B}^{n,m,d,k,\mu})\]
where $\Omega_{y}^{B} = \{(d,y,2)\in \mathbb{N}^3\,:\, d\geq 2\} \cup \{(d,k,\mu)\in \mathbb{N}^3\,:\,(\mu-1)k = y,\;\mu\leq d\leq 2\mu-3\}$ and $\eta$ is defined in~\eqref{eq:eta-def}. We refer the reader to Remark~\ref{remark: explanation of degree restriction on degree restricted hyperzouts} for the choice of such constraints on the parameters involved in $\Omega_y^B$.
Let
\begin{equation}
    \label{eq: Tilde OmegayB}
    \tilde{\Omega}_{y}^B = \{(d,k,\mu)\in {\Omega}_y^B\;:\; \mu\geq 3\}
\end{equation}
and define
\begin{equation}
\tilde{\mathcal{B}}_{m,2y}=\bigcup_{(d,k,\mu)\in \tilde{\Omega}_y^B}\bigcup_{n\geq 1}  \eta \left(\mathcal{S}_{e,B}^{n,m,d,k,\mu} \right).
\label{eq: Tilde Bezoutian collection definition}    
\end{equation}
The set $\tilde{\mathcal{B}}_{m,2y}$ contains all degree-restricted hyperzouts that are not sums of squares. The argument is very similar to the proof of Lemma~\ref{lem:Wtilde-property}, in the Wronskian setting.

\begin{lemma}
\label{lem:Btilde-property}
    Let $\tilde{\mathcal{B}}_{m,2y}$ be defined in~\eqref{eq: Tilde Bezoutian collection definition}. Then, $\mathcal{B}_{m,2y}\setminus \Sigma_{m,2y}\subseteq \tilde{\mathcal{B}}_{m,2y}$.
\end{lemma}
\begin{proof}
    Let $q\in \mathcal{B}_{m,2y}\setminus \Sigma_{m,2y}$. Then 
    \[q\in \mathcal{B}_{m,2y} = \tilde{\mathcal{B}}_{m,2y} \cup \left( \bigcup_{\substack{n\geq 1\\d\geq 2}} \eta\left(\mathcal{S}_{e,B}^{n,m,d,y,2}\right)\right)\subseteq \tilde{\mathcal{B}}_{m,2y} \cup \Sigma_{m,2y},\]
    where the inclusion follows from Lemma~\ref{lem:sos-hyperwon-hyperzout}. Since $q\notin \Sigma_{m,2y}$ by assumption, it follows that $q\in \tilde{\mathcal{B}}_{m,2y}$.
\end{proof}
Our aim, now, is to construct a set $\Gamma_{\mathcal{B}_{m,2y}}$ that is a finite union of images of semi-algebraic maps, (i.e., of the form~\eqref{eq:gamma-structure}) that contains $\tilde{\mathcal{B}}_{m,2y}$. To do this we use the observation that the map $\eta$, defined in~\eqref{eq:eta-def} factors through a (low-dimensional) space, the dimension of which is independent of $n$. 

In the argument that follows, if $a(x)\in T_{d+1,k}^{m,d+1}$ and $b(x)\in T_{d,k}^{m,d}$ then we can think of the entries $a_i(x)\in F_{m,(d-i)k}$ (for $i=0,1,\ldots,d$) as coefficients of a univariate polynomial
\[ p_a(t) = a_0(x) + a_1(x)t + \cdots + a_{d-1}(x)t^{d-1} + a_d(x)t^d\]
and the entries $b_j(x)\in F_{(m,d-1-j)k}$ (for $j=0,1,\ldots,d-1$) as the coefficients of a univariate polynomial
\[ p_b(t) = b_0(x) + b_1(x)t + \cdots + b_{d-2}(x)t^{d-2} +b_{d-1}(x)t^{d-1}.\]
Recall that we use the notation $B_d(a,b)$ to denote the Bezoutian of these two univariate polynomials.

\begin{lemma}
    \label{lem:sa-cover-B}
    Let $m$ and $y$ be positive integers and let $(d,k,\mu)\in \Omega_{y}^B$. Then 
    \[ \eta\left(\mathcal{S}_{e,B}^{n,m,d,k,\mu}\right) \subseteq \eta_1\left(T_{d+1,k}^{m,d+1}\times T_{d,k}^{m,d}\times T_{\mu,k}^{m,d}\right)\]
    where $\eta_1:T_{d+1,k}^{m,d+1}\times T_{d,k}^{m,d}\times T_{\mu,k}^{m,d}\rightarrow F_{m,2k(\mu-1)}$, defined by 
    $\eta_1(a,b,\xi) = \xi(x)^\intercal B_d(a(x),b(x)) \xi(x)$, is a semi-algebraic map.
\end{lemma}
\begin{proof}
    The map $\eta$ factors as $\eta = \eta_1\circ \eta_2$ where
    $\eta_2:F_{n,d}\times \R^n\times \R^n\times F_{m,k}^n \times T_{\mu,k}^{m,d}\rightarrow T_{d+1,k}^{m,d+1}\times T_{d,k}^{m,d}\times T_{\mu,k}^{m,d}$ is defined by
    \[ \eta_2(p,u,v,\phi,\xi) = (p_{\phi(x),u},D_vp_{\phi(x),u},\xi),\]
    where $p_{\phi(x),u}(t) = p(\phi(x)+tu)$ and $D_vp_{\phi(x),u}(t) = D_vp(\phi(x)+tu)$. Note that the coefficients of powers of $t$ in these polynomials
    can be thought of as elements of $T_{d+1,k}^{m,d+1}$ and $T_{d,k}^{m,d}$, respectively.
    
    Since $\eta_2\left(\mathcal{S}_{e,B}^{n,m,d,k,\mu}\right)\subseteq T_{d+1,k}^{m,d+1}\times T_{d,k}^{m,d}\times T_{\mu,k}^{m,d}$, it follows directly that 
\[ \eta\left(\mathcal{S}_{e,B}^{n,m,d,k,\mu}\right) \subseteq \eta_1\left(T_{d+1,k}^{m,d+1}\times T_{d,k}^{m,d}\times T_{\mu,k}^{m,d}\right).\]
    To see that $\eta_1$ is a semi-algebraic map, we note that its graph is 
    \[ \{(a,b,\xi,q)\in T_{d+1,k}^{m,d+1}\times T_{d,k}^{m,d} \times T_{\mu,k}^{m,d}(d-1)k \times  F_{m,2y}\;:\; q(x) = \xi(x)^\intercal B_d(a(x),b(x))\xi(x)\}.\]
    This is a semi-algebraic set since it is defined by the common solution of a finite collection of quartic equations, obtained by equating the coefficients on the left and right hand sides of the polynomial identity $q(x) = \xi(x)B_d(a(x),b(x))\xi(x)$. 
\end{proof}

We are now in a position to give a sufficient condition that implies the existence of a non-negative polynomial that is not a degree-restricted hyperzout.
\begin{theorem}
\label{theorem: Dimension bound of Bezoutian from hyperbolic polynomial}
Suppose $m>2$, $2y>2$ and $(m,2y)\neq (3,4)$. There exists a non-negative homogeneous polynomial of degree $2y$ in $m$ variables that is not a degree-restricted hyperzout whenever
\begin{align}
        &\dim{P_{m,2y}}=\binom{2y+m-1}{m-1} \notag\\ 
        &>\max_{(d,k,\mu)\in \Tilde{\Omega}_{y}^B} \sum_{i=0}^{\mu-1} \binom{m+ik-1}{m-1}+\sum_{i=0}^{d} \binom{m+ik-1}{m-1}+\sum_{i=0}^{d-1} \binom{m+ik-1}{m-1}.
        \label{eq: Bezoutian inequality}
\end{align}
where  
$\Tilde{\Omega}_{y}^B$ is defined in~\eqref{eq: Tilde OmegayB}. 
\end{theorem}
\begin{proof}
    Let $\Gamma_{\mathcal{B}_{m,2y}}\subseteq F_{m,2y}$ be defined by
\begin{equation*}
    \label{eq:gamB-def}
    \Gamma_{\mathcal{B}_{m,2y}} = \bigcup_{(d,k)\in \tilde{\Omega}_y^B} \eta_1\left(T_{d+1,k}^{m,d+1}\times T_{d,k}^{m,d}\times T_{\mu,k}^{m,d}\right).
\end{equation*}
We first claim that $\Gamma_{\mathcal{B}_{m,2y}} \supseteq \tilde{\mathcal{B}}_{m,2y} \supseteq \mathcal{B}_{m,2y}\setminus \Sigma_{m,2y}$. 
This holds because 
\begin{align*} \Gamma_{\mathcal{B}_{m,2y}} & = \bigcup_{(d,k)\in \tilde{\Omega}_y^B} \eta_1\left(T_{d+1,k}^{m,d+1}\times T_{d,k}^{m,d}\times T_{\mu,k}^{m,d}\right)\\
& = \bigcup_{(d,k)\in \tilde{\Omega}_y^B}\bigcup_{n\geq 1} \eta_1\left(T_{d+1,k}^{m,d+1}\times T_{d,k}^{m,d}\times T_{\mu,k}^{m,d}\right)\\
& \supseteq \bigcup_{(d,k)\in \tilde{\Omega}_y^B}\bigcup_{n\geq 1} \eta\left(\mathcal{S}_{e,B}^{n,m,d,k}\right)\\
& = \tilde{\mathcal{B}}_{m,2y} \supseteq \mathcal{B}_{m,2y}\setminus \Sigma_{m,2y}
\end{align*}
where the first equality is the definition of $\Gamma_{\mathcal{B}_{m,2y}}$, the second equality holds because $\eta_1\left(T_{d+1,k}^{m,d+1}\times T_{d,k}^{m,d}\times T_{\mu,k}^{m,d}\right)$ is independent of $n$, the first containment follows from Lemma~\ref{lem:sa-cover-B}, and the final containment follows from Lemma~\ref{lem:Btilde-property}.

To complete the proof, we assume that  $\mathcal{B}_{m,2y}= P_{m,2y}$ and derive a contradiction. Lemma~\ref{lem:sa-cover-B} shows $\eta_1$ is a semi-algebraic map, $\Tilde{\Omega}_y^B$ is a finite set construction. By the definition of $\Gamma_{\mathcal{B}_{m,2y}}$, we have $ \Gamma_{\mathcal{B}_{m,2y}} \subseteq \bigcup_{(d,k)\in \tilde{\Omega}_y^B} \eta_1\left(T_{d+1,k}^{m,d+1}\times T_{d,k}^{m,d}\times T_{\mu,k}^{m,d}\right)$.
Since $\dim(T_{\mu,k}^{m,d}) = \sum_{i=0}^{\mu-1}\binom{m+ik-1}{m-1}$, Proposition~\ref{proposition: dimension inequality non-inclusion} shows that 
\begin{align} 
\dim(\Gamma_{\mathcal{B}_{m,2y}}) &\leq \max_{(d,k,\mu)\in \tilde{\Omega}_{y}^B} \dim(T_{d+1,k}^{m,d+1}\times T_{d,k}^{m,d}\times T_{\mu,k}^{m,d})\nonumber\\
&= \max_{(d,k,\mu)\in \tilde{\Omega}_{y}^B}\sum_{i=0}^{\mu-1}\binom{m+ik-1}{m} + \sum_{i=0}^{d}\binom{m+ik-1}{m-1} + \sum_{i=0}^{d-1}\binom{m+ik-1}{m-1}\nonumber\\
& < \binom{m+2y-1}{m-1}\label{eq:dim-ineq2},
\end{align}
where the last inequality assumes ~\eqref{eq: Bezoutian inequality} holds.

Observe that if $\mathcal{B}_{m,2y} = P_{m,2y}$ then $\mathcal{B}_{m,2y}\setminus \Sigma_{m,2y} = P_{m,2y}\setminus \Sigma_{m,2y}$. On the other hand, we have established
\[\Gamma_{\mathcal{B}_{m,2y}} \supseteq \mathcal{B}_{m,2y}\setminus \Sigma_{m,2y} = P_{m,2y}\setminus \Sigma_{m,2y}.\]
This, together with Lemmas~\ref{lemma: subset-dimension} and~\ref{lemma: dimension equality} (and the assumption that $m>2$, $2y>2$ and $(m,2y)\neq (3,4)$) implies the inequality 
\[
 \dim(\Gamma_{\mathcal{B}_{m,2y}}) \geq \dim(P_{m,2y}\setminus \Sigma_{m,2y}) = \binom{m+2y-1}{m-1},\]
contradicting~\eqref{eq:dim-ineq2}. Therefore there must exist a non-negative polynomial that is not a degree-restricted hyperzout.
\end{proof}
 We conclude this subsection by discussing the degree restriction that we impose in the definition of degree-restricted hyperzout.
\begin{remark}
   The restriction of $d\leq 2\mu -3$ in Definition~\ref{definition: Bezoutian certificates} comes from the fact that~\eqref{eq: Bezoutian inequality} cannot be satisfied unless $dk < 2y$ for all $(d,k,\mu)\in \tilde{\Omega}_{y}^B$. Since $\tilde{\Omega}_y^B$ consists of tuples $(d,k,\mu)$ such that $(\mu-1)k = y$, it follows that $2y-dk = k(2\mu-2-d)>0$ for any such tuple. Recognizing that $d,k,\mu$ are all positive integers, we obtain $d\leq 2\mu-3$.
   This restriction essentially says that we do not allow the use of high degree hyperbolic polynomials to generate relatively low-degree non-negative hyperzouts. 
    \label{remark: explanation of degree restriction on degree restricted hyperzouts}
\end{remark}

\subsection{Non-negative polynomials that are not hyperwrons}

\label{sec:main-results}
In this section, we prove Theorem~\ref{theorem: m_y no Wronskian}, which gives conditions on the degree $2y$ and number of variables $m$ that ensure there exists a non-negative homogeneous polynomial that is not a hyperwron.

This result is obtained by using the sufficient condition given in Theorem~\ref{theorem: Dimension bound of Wronskian from hyperbolic polynomial}.

The main effort in the proof is then to find ranges of integer parameters where certain expressions involving binomial coefficients are non-negative. 

Before proving Theorem~\ref{theorem: m_y no Wronskian}, we establish a useful lemma about binomial coefficients.
\begin{lemma}
\label{lem:binom-ineq}
Let $F:\N^2\rightarrow \N$ be defined by $F(\ell,\alpha) = \binom{\ell+\alpha}{\ell}$. 
If $1\leq \ell' < \ell$ and $0\leq \alpha <  \beta$, then
\begin{enumerate}[(i)]
    \item $\frac{F(\ell',\alpha)}{F(\ell',\beta)} < 1$ and \label{lem:binom-ineq - 1}
    \item  $\frac{F(\ell',\alpha)}{F(\ell,\alpha)} > \frac{F(\ell',\beta)}{F(\ell,\beta)}$. \label{lem:binom-ineq - 2}
\end{enumerate}
\end{lemma}
\begin{proof}
For the first inequality we have that 
\[ \frac{F(\ell',\alpha)}{F(\ell',\beta)} =\frac{(\ell'+\alpha)(\ell'-1 + \alpha) \cdots (1+\alpha)}{(\ell'+\beta)(\ell'-1+\beta) \cdots (1+\beta)} = \prod_{j=1}^{\ell'}\left(\frac{j+\alpha}{j+\beta}\right) < 1\]
since $\alpha<\beta$. Similarly, 
    \begin{equation*}
\label{eq:binom-ineq}
\frac{F(\ell',\alpha)}{F(\ell',\beta)} = 
 \prod_{j=1}^{\ell'}\left(\frac{j+\alpha}{j+\beta}\right)> \prod_{j=1}^{\ell}\left(\frac{j+\alpha}{j+\beta}\right) = \frac{F(\ell,\alpha)}{F(\ell,\beta)}.\end{equation*}
since each term in the product is strictly less than one, and $\ell > \ell'$.
\end{proof}

We now proceed with the proof of Theorem~\ref{theorem: m_y no Wronskian}.

\begin{proof}[{Proof of Theorem~\ref{theorem: m_y no Wronskian}}]
Our aim is to find positive integer values of $m$ and $y$ such that~\eqref{eq: hyperbolic-Wronskian inequality} holds, i.e.,
\[ \binom{2y+m-1}{m-1} > \max_{(d,k)\in \tilde{\Omega}_y^W} 2\binom{(d-1)k+m-1}{m-1} + \binom{dk+m-1}{m-1} + \binom{(d-2)k+m-1}{m-1}\]
where $\tilde{\Omega}_y^W = \{(d,k)\in \mathbb{N}^2\;:\;d \geq 3, (d-1)k=y\}$. This is implied by  
\[ \binom{2y+m-1}{m-1} > \max_{2\leq 2k\leq y} 2\binom{y+m-1}{m-1} + \binom{y+k+m-1}{m-1} + \binom{y-k+m-1}{m-1},\]
which is obtained by eliminating $d$ and noting that any $k$ such that $(d,k)\in \tilde{\Omega}_y^W$ satisfies $y = (d-1)k \geq 2k$. 

Let 
\[ g(m,k,y) = \binom{2y+m-1}{m-1} -2\binom{y+m-1}{m-1} - \binom{y+k+m-1}{m-1} - \binom{y-k+m-1}{m-1}.\]
We first consider the case $m=4$. An explicit computation gives 
\[ g(4,k,y) = 2y^3/3 - k^2y-11 y/3 - 2k^2 - 3\]
which is decreasing for increasing $|k|$. Therefore if $2\leq 2k\leq y$ we have that 
\[ g(4,k,y) \geq g(4,y/2,y) = 5y^3/12 - y^2/2 - 11y/3 - 3.\]
It is straightforward to check that $g(4,y/2,y)>0$ whenever $y\geq 4$. This shows that there is a non-negative polynomial of degree $2y\geq 8$ in $m=4$ variables that is not a hyperwron.

We deal with the case $m=5$ via a similar approach. An explicit computation gives
\[ g(5,k,y) = y^4/2 + 5 y^3/3 -k^2 y^2/2 -5 k^2 y/2 -25 y/6 -k^4/12 -35 k^2/12 -3\]
which is decreasing for increasing $|k|$.
Therefore if $2\leq 2k\leq y$ we have that 
\[ g(5,k,y) \geq g(5,y/2,y) = 71 y^4/192+25 y^3/24-35 y^2/48-25 y/6-3. \]
It is straightforward to check that $g(5,y/2,y)= 0$ when $y= 2$ and that $g(5,y/2,y) > 0$ whenever $y \geq 3$. This shows that there is a non-negative polynomial of degree $2y\geq 6$ in $m=5$ variables that is not a hyperwron.
 
Next, we consider the case of general $m$. We will show that if $0<k<y$, $2\leq m'<m$ and $g(m',k,y)\geq 0$, then $g(m,k,y)>g(m',k,y)$. Before establishing this, we see how it completes the proof for $m\geq 6$. Indeed, it then follows that 
when $1\leq k\leq y/2$ and $m\geq 6$ we have 
\[ g(m,k,y) > g(5,k,y).\]
Since $g(5,k,y)\geq 0$ whenever $1\leq k\leq y/2$ and $y\geq 2$, it follows that $g(m,k,y) >0$ whenever $m\geq 6$ and $1\leq k \leq y/2$ and $y\geq 2$. This implies that whenever $m\geq 6$ and $y\geq 2$, there is a non-negative polynomial that is not a hyperwron.

It remains to establish that $g(m,k,y)>g(m',k,y)$ when $0<k<y$ and $2\leq m'<m$. 

If $0< k <y$, then $0<y<2y$ and $0<y\pm k < 2y$. Therefore, given $0< k < y$ and adopting the definition of $F(\cdot,\cdot)$ as in Lemma~\ref{lem:binom-ineq}, we have
\begin{align*}
      g(m,k,y) &=F(2y,m-1) - F(y+k,m-1) - 2F(y,m-1) -F(y-k,m-1)\\
      &>F(2y,m-1)\left(1-\frac{F(y+k,m'-1)}{F(2y,m'-1)}-2\frac{F(y,m'-1)}{F(2y,m'-1)} -\frac{F(y-k,m'-1)}{F(2y,m'-1)} \right)\\
      &=\frac{F(2y,m-1)}{F(2y,m'-1)}g(m',k,y)\\
      &>g(m',k,y),
\end{align*}
where the first inequality follows from Lemma~\ref{lem:binom-ineq} (ii) and the second inequality follows from Lemma~\ref{lem:binom-ineq} (i) and $g(m',k,y)\geq 0$. This completes the proof.

\end{proof}

\subsection{Non-negative polynomials that are not degree-restricted hyperzouts}
We now turn our attention to showing the existence of non-negative polynomials that are not degree-restricted hyperzouts. Our sufficient condition (Theorem~\ref{theorem: Dimension bound of Bezoutian from hyperbolic polynomial}) for this is less refined than its counterpart for hyperwrons (Theorem~\ref{theorem: Dimension bound of Wronskian from hyperbolic polynomial}). As such, we only aim to show that given any integer $y>1$, for a sufficiently large number of variables $m$ there is an element of $P_{m,2y}$ that is not a degree-restricted hyperzout.
\begin{theorem}
\label{theorem: m_y no Bezoutian}
If $m,y$ are positive integers such that $y>1$ and $m>10y^2-2y+1$, then there exists a non-negative homogeneous polynomial in $m$ variables of degree $2y$ that is not a degree-restricted hyperzout.

\end{theorem}
\begin{proof}
Our strategy will be to show that inequality~\eqref{eq: Bezoutian inequality} follows from the assumption that $m>10y^2-2y+1$.
We first note that if $y>1$, then $m>10y^2-2y+1$ implies that $2y>2$, $m>2$ and $(m,2y)\neq (3,4)$.

Dividing by $\binom{2y+m-1}{m-1}$, we have that~\eqref{eq: Bezoutian inequality} is equivalent to
\[1>J\coloneqq \frac{\sum_{i=0}^{\mu-1} \binom{m+ik-1}{m-1}}{\binom{2y+m-1}{m-1}}+\frac{\sum_{i=0}^{d} \binom{m+ik-1}{m-1}}{\binom{2y+m-1}{m-1}}+\frac{\sum_{i=0}^{d-1}\binom{m+ik-1}{m-1}}{\binom{2y+m-1}{m-1}}.\] 

From Lemma~\ref{lem:binom-ineq}, since $\binom{m+ik-1}{m-1}$ increases monotonically with $k\in \N$, we can bound each sum of the form $\sum_{i=0}^{\ell}\binom{m+ik-1}{m-1}$ above by $(\ell+1)\binom{m+i\ell-1}{m-1}$. This gives
\begin{equation*}
   J\leq \frac{\mu\binom{m-1+y}{m-1}+(2d+1)\binom{m-1+dk}{m-1}}{\binom{m-1+2y}{m-1}}.
   \end{equation*}

  From Remark~\ref{remark: explanation of degree restriction on degree restricted hyperzouts}, we know that $dk\leq 2y-1$. Since, in addition, $y\leq 2y-1$ for any positive integer $y$, we have 
   \begin{equation*}
       J\leq \frac{\mu\binom{m-1+y}{m-1}+(2d+1)\binom{m-1+dk}{m-1}}{\binom{m-1+2y}{m-1}} \leq 
\frac{(2d+\mu+1)\binom{m-1+(2y-1)}{m-1}}{\binom{m-1+2y}{m-1}}
   =\frac{(2d+\mu+1)2y}{m-1+2y}.
\end{equation*}
To complete the proof, it suffices to show that $m>10y^2-2y+1$ implies $\frac{(2d+\mu+2)2y}{m-1+2y}<1$. We proceed by recognizing that $d\leq dk\leq 2y-1$ and $y = (\mu-1)k \geq \mu-1$ for $k\geq 1$. Therefore,
\begin{align*}
    \frac{(2d+\mu+1)2y}{m-1+2y} \leq \frac{10y^2}{m-1+2y} < 1
\end{align*}
whenever $m > 10y^2-2y+1$.

\end{proof}

%% file: Main_Body/Not_Hyperwron_Example.tex
\section{A non-negative quartic that is not a sum of hyperwrons}
\label{section: An example non-negative quartic homogeneous polynomial that is not hyperwron}

Theorem~\ref{theorem: m_y no Wronskian} shows that there are non-negative polynomials that are not hyperwrons by showing that the non sum-of-squares components of the set of hyperwrons form a set that is not full dimensional in the ambient space of non-negative polynomials. 

Any sum of hyperwrons is, of course, still a non-negative polynomial. It is, therefore, reasonable to ask whether we get more non-negative polynomials by considering sums of hyperwrons, rather than just hyperwrons. 

To formalise this question, consider the conic hull of hyperwrons, $\textup{cone}(\mathcal{W}_{m,2y})$. This is a convex cone lying between the cone of sums of squares and the cone of all non-negative polynomials. By Carath\'eodory's theorem, any element of $\textup{cone}(\mathcal{W}_{m,2y})$ is the sum of at most  $r=\dim(\textup{cone}(\mathcal{W}_{m,2y})) = \binom{m+2y-1}{2y}$ extreme elements of $\textup{cone}(\mathcal{W}_{m,2y})$. For any set $S$, the extreme rays of $\textup{cone}(S)$ are generated by elements of $S$~\cite[Corollary 17.1.2]{Rockafellar1970ConvexAnalysis}. Therefore any element of $\textup{cone}(\mathcal{W}_{m,2y})$ is a sum of at most $r$ hyperwrons.

Since the conic hull of hyperwrons is full-dimensional, the techniques used to prove Theorem~\ref{theorem: m_y no Wronskian} do not rule out the possibility that all non-negative polynomials are sums of hyperwrons. In this section, however, we show that this is false by giving an explicit example of a non-negative quartic form in 16 variables that is not a sum of hyperwrons. To show that our example is not a sum of hyperwrons, we show that it is not a hyperwron (Theorem~\ref{theorem: example polynomial is not a hyperwron}) and that it generates an extreme ray in the cone of non-negative polynomials (Proposition~\ref{proposition: example polynomial is an extreme ray}). Together, by appealing to the definition of extreme rays, these observations imply that our example is not a sum of hyperwrons.

\subsection{Structure of degree four hyperwrons}

In this section we consider the structure of hyperwrons of degree four. For a hyperwron of degree four, there are only two combinations of the degree $d$ of hyperbolic polynomial $p$ and the degree $k$ of the map $\phi$ that can be used in the construction, namely where $d=2,k=2$ and $d=3,k=1$. The $d=2$ case corresponds to hyperwrons that are sums of squares. 
We will show that hyperwrons generated by hyperbolic polynomials of degree $d=3$ can be written in a particular structured form (see Theorem~\ref{theorem: Hyperwron representation of cubic hyperbolic polynomial, linear map}) that provides a potential obstruction to being a hyperwron. 

We begin with a technical observation about the directional derivatives of cubic hyperbolic polynomials.

\begin{proposition}
\label{proposition: directional derivative of cubic hyperbolic polynomial}
    Let $p\in \textup{Hyp}_{n,3}(e)$, and let $u,v\in \Lambda_+(p,e)$. Either $D_{uv}^2p(x)=0$ for all $x$ or there exist $q\in \Sigma_{n,2}$ and $\alpha \in F_{n,1}$ such that $D_up(x)=-q(x)+\alpha(x)D_{uv}^2p(x)$.
\end{proposition}
\begin{proof}
Since $p(x)$ is cubic, $D_up(x)$ is either quadratic or identically zero. If $D_up(x)$ is identically zero, then the result holds trivially. So assume that $D_up(x)$ is not identically zero. It is a basic fact about quadratic hyperbolic polynomials~\cite[p.~958]{Garding1959AnPolynomials} that any quadratic hyperbolic polynomial can be written in the form
$D_up(x)=\langle a_1,x\rangle^2-\sum_{i=2}^n \langle a_i, x\rangle^2$ where $a_i\in \R^n$. 
 
For convenience of notation, let $w_i = \langle a_i,v\rangle$ for $i=1,2,\ldots,n$. 

Since $u,v\in \Lambda_+(p,e)$ it follows from Proposition~\ref{prop: boundary of hyperbolicity cone directional derivative relaxation} that $\Lambda_+(D_u p,e)\supseteq \Lambda_+(p,e)$ and so that $v\in \Lambda_+(D_up,e)$. Therefore $D_up(v) = w_1^2-\sum_{i=2}^n w_i^2\geq 0$. 
We consider two cases: either $w_1 = 0$ or $w_1 \neq 0$.

\textbf{Case 1: $w_1\neq 0$.} In this case,

\begin{align*}
   q(x):=\sum_{i=2}^n \langle a_i,x\rangle^2 - \frac{(\sum_{i=2}^n \langle a_i,x\rangle w_i)^2}{w_1^2} & \geq 
     \sum_{i=2}^n \langle a_i,x\rangle^2-\frac{(\sum_{i=2}^n \langle a_i,x\rangle w_i)^2}{\sum_{i=2}^n w_i^2}\\
     & = \frac{(\sum_{i=2}^n\langle a_i,x\rangle^2)(\sum_{i=2}^n w_i^2)-(\sum_{i=2}^n \langle a_i,x\rangle w_i)^2}{\sum_{i=2}^n w_i^2}\\
     & \geq 0,
\end{align*}
where the first inequality uses $w_1^2 - \sum_{i=2}^nw_i^2 \geq 0$ and the last inequality comes from the Cauchy-Schwarz inequality.

We have established that $q$ is a globally non-negative quadratic form, and hence $q$ is a sum of squares.

We aim to represent $D_up(x)$ as $-q(x) + \alpha(x)D_{uv}^2p(x)$ for some linear form $\alpha$.
Note that 
\[ D_{uv}^2p(x) = 2w_1\langle a_1,x\rangle - 2\sum_{i=2}^{n}w_i\langle a_i,x\rangle.\]
We write
\begin{align*}
w_1^2\langle a_1,x\rangle^2
    &=\left(\langle a_1,x\rangle w_1+\sum_{i=2}^n\langle a_i,x\rangle 
 w_i \right)\left(\langle a_1,x\rangle 
 w_1 -\sum_{i=2}^n\langle a_i,x\rangle w_i \right)+\left(\sum_{i=2}^nw_i\langle a_i,x\rangle\right)^2\\
    &=\left(\langle a_1,x\rangle w_1+\sum_{i=2}^n\langle a_i,x\rangle w_i \right)\frac{1}{2}D_{uv}^2p(x)+\left(\sum_{i=2}^nw_i\langle a_i,x\rangle\right)^2.
\end{align*}
Let $\alpha(x) = \frac{1}{2w_1^2}\left(\langle a_1,x\rangle w_1+\sum_{i=2}^n\langle a_i,x\rangle w_i \right)$.
Therefore,
\begin{align*}
    D_up(x) & =\langle a_1,x\rangle^2-\sum_{i=2}^n \langle a_i , x\rangle^2\\
    & = \alpha(x) D_{uv}^2p(x) - \sum_{i=2}^n \langle a_i , x\rangle^2 + \frac{1}{w_1^2}\left(\sum_{i=2}^nw_i\langle a_i,x\rangle\right)^2\\
    & = \alpha(x)D_{uv}^2p(x) - q(x).
\end{align*}

\textbf{Case 2: $w_1=0$.} In this case, since $w_1^2\geq \sum_{i=2}^n w_i^2$ it follows that  $w_i=0$ for $1\leq i\leq n$. Then $D_up(v) = 0$ and 
\[ D_{uv}^2p(x) = 2w_1\langle a_1,x\rangle - 2\sum_{i=2}^{n}w_i\langle a_i,x\rangle = 0 \quad\textup{for all $x$},\]
completing the proof. 

\end{proof}
The following simple fact about quadratic forms will be useful in what follows.
\begin{lemma}
\label{lem:qform-fact}
    Let $q\in F_{n,2}$ be a quadratic form and let $l\in F_{n,1}$ be a linear form. If $q(x)\geq 0$ whenever $l(x)=0$, then there exists a sum of squares $s\in \Sigma_{n,2}$ and a linear form $\alpha\in F_{n,1}$ such that 
    \[ q(x) = s(x) + l(x)\alpha(x)\quad\textup{for all $x\in \R^n$}.\]
\end{lemma}
\begin{proof}
If $l\in F_{n,1}$ is identically zero, then $q(x) \geq 0$ for all $x$, and so $q(x)=s(x)$ is a sum of squares. 

Next, we assume that $l$ is not identically zero.
Let $Q$ be a symmetric matrix such that $x^\intercal Q x = q(x)$ for all $x$. Let $\ell\in \R^n\setminus\{0\}$ be such that $\ell^\intercal x = l(x)$ for all $x$ and let $\hat{\ell} = \ell/\|\ell\|$ be the corresponding unit vector.
    Let $L = \{x\in \R^n\,|\,l(x)=0\}$ and let $P_L$ and $P_{L^\perp}$ denote the orthogonal projectors onto $L$ and $L^\perp$, respectively. Note that $P_{L^\perp}x = \hat{\ell}(\hat{\ell}^\intercal x)$. Moreover, the assumption that $q(x)\geq 0$ for all $x\in L$ is equivalent to $s(x):= x^\intercal P_L Q P_L x \geq 0$ for all $x$. Therefore $s$ is a sum of squares.
    
    Since $P_L+P_{L^\perp} = I$, 
    \begin{align*}
 x^\intercal Q x & = x^\intercal P_L Q P_L x + x^\intercal P_{L^\perp} Q P_{L}x + x^\intercal P_L Q P_{L^\perp} x + x^\intercal P_{L^\perp}QP_{L^\perp}x\\
 & = x^\intercal P_L Q P_L x + (\hat{\ell}^\intercal x)(\hat{\ell}^\intercal QP_{L}x) + (\hat{\ell}^\intercal QP_L x)(\hat{\ell}^\intercal x) + (\hat{\ell}^\intercal x)^2\hat{\ell}^\intercal Q \hat{\ell}\\
 & = s(x) + (\ell^\intercal x)\frac{1}{\|\ell\|}\left[2\hat{\ell}^\intercal QP_L x + (\hat{\ell}^\intercal Q\hat{\ell}) \hat{\ell}^\intercal x\right]\\
  & = s(x) + \frac{l(x)}{\|\ell\|}\left[2\hat{\ell}^\intercal QP_L x + (\hat{\ell}^\intercal Q\hat{\ell}) \hat{\ell}^\intercal x\right].
    \end{align*}
    This has the desired form, completing the proof. 
\end{proof}

We now use the result of Proposition~\ref{proposition: directional derivative of cubic hyperbolic polynomial} to show that the Wronskians of hyperbolic cubics can be written in a particular form. Indeed, for each Wronskian $f$ of a hyperbolic cubic, there exists a codimension one subspace such that $f$ is a product of sums of squares when restricted to that subspace.
\begin{proposition}
    \label{proposition: Hyperwron representation of cubic hyperbolic polynomial}
    Let $p\in \textup{Hyp}_{n,3}(e)$, and let $u,v\in \Lambda_+(p,e)$. Then there exist sums of squares $q_1,q_2\in \Sigma_{n,2}$, a linear form $l\in F_{n,1}$, and a cubic form $r \in F_{n,3}$, such that the Wronskian $f(x)=D_up(x)D_vp(x)-p(x)D_{uv}^2p(x)$ can be represented as 
    \[f(x)=q_1(x)q_2(x)+r(x)l(x).\] Moreover, if $D_{uv}^2p(x)\neq 0$, one can take $l(x)=D_{uv}^2p(x)$.
\end{proposition}
\begin{proof}
\textbf{Case 1:} Suppose $D_{uv}^2p(x)=0$ for all $x\in \R^n$. Then 
$f(x)=D_up(x)D_vp(x)$. If $D_up(x)$ is identically zero, then $f$ trivially has the desired form, so we assume that $D_up(x)$ is not identically zero. 
Since $D_up(x)$ is a quadratic hyperbolic polynomial (by Proposition~\ref{prop: boundary of hyperbolicity cone directional derivative relaxation}), it can be expressed in the form
\[ D_up(x) = \langle a_1,x\rangle^2 - \sum_{i=2}^{n}\langle a_i,x\rangle^2\]
for some $a_1,a_2,\ldots,a_n\in \R^n$, at least one of which is non-zero~\cite[p.958]{Garding1959AnPolynomials}. Therefore, if we let $q_1(x) = \sum_{i=2}^{n}\langle a_i,x\rangle^2$, we have
\begin{equation}
    \label{eq:Duvzero-case}
f(x) = D_up(x)D_vp(x) = q_1(x)\left(-D_vp(x)\right) + \langle a_1,x\rangle \left(\langle a_1,x\rangle D_vp(x)\right).
\end{equation} 
Let $a_1^\perp \coloneqq\{x\in \R^n:\langle a_1,x\rangle=0\}$. We consider cases according to the behaviour of $q_1(x)$ when restricted to $a_1^\perp$.

\textbf{Case 1a:} Suppose $q_1$ is identically zero when restricted to $a_1^\perp$. This implies that $q_1(x) = \langle a_1,x\rangle \alpha(x)$ for some linear form $\alpha\in F_{n,1}$. Then, 
\[ f(x) = \langle a_1,x\rangle (\langle a_1,x\rangle-\alpha(x)) D_{v}p(x).\]
This is in the desired form with $l(x) = \langle a_1,x\rangle$ and $r(x) = (\langle a_1,x\rangle - \alpha(x)) D_vp(x)$.

\textbf{Case 1b:} Otherwise, assume that the restriction of $q_1$ to $a_1^\perp$ is not identically zero. 
Since $f$ is a hyperwron, it is non-negative. It follows from~\eqref{eq:Duvzero-case} that $f(x) = q_1(x)(-D_vp(x)) \geq 0$ whenever $x\in a_1^\perp$. Next, we will show that $-D_vp(x) \geq 0$ whenever $x\in a_1^\perp$. To do this, we argue by contradiction. Suppose $-D_vp(x) < 0$ for some $x\in a_1^\perp$. Then, by continuity, there exists $\epsilon>0$ such that $-D_vp(y) <0$ for all $y\in \mathcal{B}_x(\epsilon) = \{y\in a_1^\perp\;:\;\|y-x\|< \epsilon\}$.
Since $q_1(x)(-D_vp(x)) \geq 0$ for all $y\in \mathcal{B}_x$ and $q_1$ is a sum of squares, it follows that $q_1(y) = 0$ for all $y\in \mathcal{N}_x$. Since $\mathcal{N}_x$ is an open subset of $a_1^\perp$, it follows that $q_1$ is identically zero on $a_1^\perp$, a contradiction. Therefore, we can conclude that $-D_vp(x) \geq 0$ whenever $x\in a_1^\perp$.

Lemma~\ref{lem:qform-fact} then tells us that 
\[ -D_vp(x) = q_2(x) + \langle a_1,x\rangle \alpha(x)\]
for some $\alpha \in F_{n,1}$ and a sum of squares $q_2\in \Sigma_{n,2}$. Overall, then, we have
\[ f(x) =q_1(x)q_2(x) + \langle a_1,x\rangle\left(\langle a_1,x\rangle D_vp(x) + \alpha(x)q_1(x)\right),\]
which has the desired form with $l(x) = \langle a_1,x\rangle$ and $r(x) = \langle a_1,x\rangle D_vp(x) + \alpha(x)q_1(x)$.

\textbf{Case 2:} Otherwise, assume that $D_{uv}^2p(x)$ is not identically zero. Proposition~\ref{proposition: directional derivative of cubic hyperbolic polynomial} asserts that there exist 
$q_1\in \Sigma_{n,2}$ and $\alpha_1\in F_{n,1}$ such that $D_up(x) = -q_1(x) + \alpha_1(x)D_{uv}^2p(x)$. Similarly, by exchanging the roles of $u$ and $v$, there exist $q_2\in \Sigma_{n,2}$ and $\alpha_2\in F_{n,1}$ such that $D_vp(x) = -q_2(x) + \alpha_2(x)D_{uv}^2p(x)$.

We can then express the Wronskian $f$ as 
\begin{align*}
    f(x)&=D_up(x)D_vp(x)-p(x)D_{uv}^2p(x)\\
    &=\left(-q_1(x)+\alpha_1(x)D_{uv}^2p(x)\right)\left(-q_2(x)+\alpha_2(x)D_{uv}^2p(x)\right)-p(x)D_{uv}^2p(x)\\
    &=q_1(x)q_2(x)+D_{uv}^2p(x) \left(-q_1(x)\alpha_2(x)-q_2(x)\alpha_1(x)+\alpha_1(x)\alpha_2(x)D_{uv}^2p(x)-p(x)  \right).
\end{align*}
The result follows by setting $r(x)=-q_1(x)\alpha_2(x)-q_2(x)\alpha_1(x)+\alpha_1(x)\alpha_2(x)D_{uv}^2p(x)-p(x)$ and $l(x)=D_{uv}^2(x)$.   
\end{proof}
We can translate this from a statement about Wronskians of hyperbolic cubics into a statement about hyperwrons of degree four.

\begin{theorem}
        \label{theorem: Hyperwron representation of cubic hyperbolic polynomial, linear map}
    Let $\tilde{f}\in \mathcal{W}_{m,4}$ be a hyperwron of degree four. Then either $\tilde{f}$ is a sum of squares or there exist sums of squares  $\Tilde{q}_1, \Tilde{q}_2\in \Sigma_{m,2}$, a linear form $\Tilde{l}\in F_{m,1}$, and a cubic form $\Tilde{r}\in F_{m,3}$ such that 
    \begin{equation}
    \label{eq: cubic hyperwron representation}
\tilde{f}(x)=\Tilde{q}_1(x)\Tilde{q}_2(x)+\Tilde{r}(x)\Tilde{l}(x).
    \end{equation}

\end{theorem}
\begin{proof}
    Any hyperwron $\tilde{f}$ of degree four is either a sum of squares or of the form $f \circ \phi$ where $f(\hat{x})= D_up(\hat{x})D_vp(\hat{x}) - p(\hat{x})D_{uv}^2p(\hat{x})$ for some $p\in \textup{Hyp}_{n,3}(e)$, $u,v\in \Lambda_+(p,e)$, and linear map $\phi:\R^m\rightarrow \R^n$.

    Proposition~\ref{proposition: Hyperwron representation of cubic hyperbolic polynomial} tells us that there exist $q_1,q_2\in \Sigma_{n,2}$ and $l\in F_{n,1}$ and $r\in F_{n,3}$ such that 
    \[ \tilde{f}(x) = q_1(\phi(x))q_2(\phi(x)) + l(\phi(x))r(\phi(x)).\]
    The result follows by setting $\Tilde{q}_1 = q_1\circ \phi$, $\Tilde{q}_2 = q_2\circ \phi$, $\tilde{l} = l\circ \phi$ and $\tilde{r} = r\circ \phi$, together with the observation that an affine change of argument preserves the degree and the property of being a sum of squares.

\end{proof}

\subsection{The example}
To identify an explicit non-negative polynomial that is not a hyperwron, we take advantage of the structure of quartic hyperwrons given in Theorem~\ref{theorem: Hyperwron representation of cubic hyperbolic polynomial, linear map}. In particular, we would like to find a non-negative quartic form that is not a sum of squares and for which no restriction to a codimension one subspace is a product of sums of squares. A challenge in doing so is the need to reason about all possible restrictions to a codimension one subspace. This is greatly simplified if the candidate in mind has a very large symmetry group.

Let $\mathbb{H}$ denote the quaternions, the four-dimensional real normed division algebra spanned by elements $1,i,j,k$, where $1$ is the multiplicative identity and $i^2 = j^2 = k^2 = ijk = -1$. If $x = a+bi+cj+dk\in \mathbb{H}$, then the \emph{real part} is $\textup{Re}(x) = a$, the \emph{conjugate} is $x^* = a-bi-cj-dk$, and the \emph{norm} of $x$ is $|x| = (a^2+b^2+c^2+d^2)^{1/2} = (xx^*)^{1/2}$. If $Z\in \mathbb{H}^{2\times 2}$ is a $2\times 2$ matrix with quaternion entries then its conjugate transpose is
\[ Z^* = \begin{bmatrix}Z_{11}^* & Z_{21}^*\\Z_{12}^* & Z_{22}^*\end{bmatrix}.\]
A quaternionic matrix $Z$ is \emph{Hermitian} if $Z = Z^*$. Note that Hermitian quaternionic matrices have real diagonal entries. If $Z$ is a $2\times 2$ Hermitian quaternionic matrix, then 
the \emph{Moore determinant} is the real number given by
\[ {\det}_M\begin{bmatrix} Z_{11} & Z_{12}\\Z_{12}^* & Z_{22}\end{bmatrix} = Z_{11}Z_{22} - Z_{12}Z_{12}^* = Z_{11}Z_{22} - |Z_{12}|^2.\]
If we write $Z_{12} = a+bi+cj+dk$, then we can think of the Moore determinant as the quadratic form $Z_{11}Z_{22} - (a^2+b^2+c^2+d^2)$ in the six real variables $Z_{11},Z_{22},a,b,c,d$. 

Let $X\in \mathbb{H}^{2\times 2}$ be a $2\times 2$ quaternionic matrix so that $XX^*$ is Hermitian. The example of a quartic form that we focus on in this section is the quartic form in $16$ real variables defined by
\begin{equation}\label{eq:quartic-eg}
\hat{f}(X) = {\det}_M(XX^*).
\end{equation}

The form $\hat{f}$ has an alternative interpretation in terms of the Cauchy-Schwarz inequality over the quaternions. If $x,y\in \mathbb{H}^k$ are vectors with quaternionic entries, then we can define
\[ \|x\|^2 = \sum_{i=1}^{k} x_ix_i^*\in \R\;\;\textup{and}\;\;\langle x,y\rangle_{\mathbb{H}} = \sum_{i=1}^{k}x_iy_i^*\in \mathbb{H}.\]
If $X\in \mathbb{H}^{2\times k}$ is the matrix of the form 
\[ X = \begin{bmatrix} x_1 & x_2 & \cdots & x_k\\y_1 & y_2 & \cdots & y_k\end{bmatrix}\]
then 
\[ {\det}_{M}(XX^*) = {\det}_{M}\begin{bmatrix} \|x\|^2 & \langle x,y\rangle_{\mathbb{H}}\\\langle x,y\rangle_{\mathbb{H}}^* & \|y\|^2\end{bmatrix} = \|x\|^2\|y\|^2 - |\langle x,y\rangle_{\mathbb{H}}|^2.\]
This form was studied in~\cite{Ge2023IsoparametricSquares}, as a special case of a broader class of isoparametric forms. Clearly $\hat{f}$, defined in~\eqref{eq:quartic-eg}, is the special case $k=2$. It also coincides with the form stated in Theorem~\ref{thm:main-sum-hyperwrons}. The following result plays an important role in this section.
\begin{theorem}[{\cite[Proposition 6.1]{Ge2023IsoparametricSquares}}]
    If $k\geq 2$ then the quartic form   $\|x\|^2\|y\|^2 - |\langle x,y\rangle_{\mathbb{H}}|^2$ in $8k$ real variables is nonnegative but not a sum of squares.
\end{theorem}
In particular, the form $\hat{f}$ defined in~\eqref{eq:quartic-eg} is nonnegative but not a sum of squares. To establish this result, Ge and Tang show that there are no nontrivial quadratic forms in $8k$ variables that vanish whenever $\hat{f}$ vanishes.

The quartic form~\eqref{eq:quartic-eg} is a particularly interesting  example of a non-negative form that is not a sum of squares because it has a very large symmetry group.
Let $\textup{Sp}(n)$ denote the group of $n\times n$ quaternionic matrices $U$ that satisfy $UU^* = U^*U = I$.
\begin{lemma}
\label{lemma: 16-variable f unitary matrix equality property}
Let $\hat{f}$ denote the quartic form defined in~\eqref{eq:quartic-eg}. If $P,Q\in \textup{Sp}(2)$ are $2\times 2$ quaternionic unitary matrices then $\hat{f}(PXQ) = \hat{f}(X)$ for all $X\in \mathbb{H}^{2\times 2}$.   
\end{lemma}
\begin{proof}
We use the representation~\eqref{eq:quartic-eg} of $\hat{f}$ in terms of the Moore determinant of Hermitian quaternionic matrices. Then
\begin{align}
      \hat{f}\left(PXQ\right)&=  {\det}_M\left(PXQ  \left(PXQ \right)^*\right)\nonumber\\
&= {\det}_M\left(PXQ Q^* X^* P^*\right)\nonumber\\
&={\det}_M\left(PX X^* P^*\right)\label{eq:Qunitary}\\
&={\det}_M(PP^*)\,{\det}_M(XX^*)\label{eq:Mdet-id}\\
& = \hat{f}(X)\label{eq:Punitary}
\end{align}
where~\eqref{eq:Qunitary} holds since $QQ^*=I$,~\eqref{eq:Mdet-id} holds due to a property of the Moore determinant~\cite[Theorem 1.1.9 (ii)]{Alesker2003Non-commutativeVariables}, and~\eqref{eq:Punitary} holds because ${\det}_M(PP^*) = {\det}_M(I) = 1$. 

\end{proof}

\subsection{The example is not a hyperwron}

In this section we show that the quartic form $\hat{f}$ defined in~\eqref{eq:quartic-eg} is not a hyperwron. First, we show that if it were a hyperwron then, by exploiting symmetry, it must have a representation in the form of~\eqref{eq: cubic hyperwron representation} where the linear form only depends on two variables.

\begin{lemma}
\label{lemma: diagonalized 16-variable quartic hyperbolic form}
Let $\hat{f}$ denote the quartic form defined in~\eqref{eq:quartic-eg}. If $\hat{f}$ is a hyperwron, then there exist sums of squares $\hat{q}_1,\hat{q}_2\in \Sigma_{16,2}$, a cubic form $\hat{r}\in F_{16,3}$ and real numbers $\sigma_1,\sigma_2\in \R$ such that 
\[ \hat{f}(X) = \hat{q}_1(X)\hat{q}_2(X) + \hat{r}(X)(\sigma_1 \textup{Re}(X_{11}) + \sigma_2 \textup{Re}(X_{22})).\]
 
\end{lemma}
\begin{proof}
Suppose that  $\hat{f}$ is a hyperwron. Since $\hat{f}$ it is not a sum of squares~\cite[Proposition 6.1]{Ge2023IsoparametricSquares}, by Theorem~\ref{theorem: Hyperwron representation of cubic hyperbolic polynomial, linear map} there exist sums of squares $\tilde{q}_1,\tilde{q}_2\in \Sigma_{16,2}$, a cubic form $\tilde{r}\in F_{16,3}$ and a linear form $\tilde{l}\in F_{16,1}$ such that $\hat{f} = \tilde{q}_1\tilde{q}_2 + \tilde{r}\tilde{l}$. 
Since $\tilde{l}$ is a linear functional, it can be expressed in the form
\[ \tilde{l}(X) = \textup{Re}\,\textup{tr}\left(A^*X\right)\]
for some matrix $A\in \mathbb{H}^{2\times 2}$. Let $A = U\textup{diag}(\sigma_1,\sigma_2) V^*$ denote the quaternionic singular value decomposition of $A$~\cite[Proposition 5.3.6 (c)]{Rodman2014TopicsAlgebra}, where $U,V\in \textup{Sp}(2)$ and $\sigma_1,\sigma_2\in \R$. By Lemma~\ref{lemma: 16-variable f unitary matrix equality property}, $\hat{f}$ is invariant under the action of $\textup{Sp}(2)$ by left- and right-multiplication. Therefore
\begin{align*}
    \hat{f}(X) &= \hat{f}(UXV^*)\\
    & = \tilde{q}_1(UXV^*)\tilde{q}_2(UXV^*) + \tilde{r}(UXV^*) \textup{Re}\,\textup{tr}(A^*U^*XV^*)\\
    & = \tilde{q}_1(UXV^*)\tilde{q}_2(UXV^*) + \tilde{r}(UXV^*) \textup{Re}\,\textup{tr}(\textup{diag}(\sigma_1,\sigma_2)X)\\
    & = \tilde{q}_1(UXV^*)\tilde{q}_2(UXV^*) + \tilde{r}(UXV^*) (\sigma_1 \textup{Re}(X_{11}) + \sigma_2\textup{Re}(X_{22})),
    \end{align*}
    where the second-last equality holds by using the fact that $A^* = V\textup{diag}(\sigma_1,\sigma_2)U^*$, the fact that $U,V\in \textup{Sp}(2)$ and the cyclic property of the trace for quaternionic matrices.

    Taking $\hat{q}_1(X) = \tilde{q}_1(UXV^*)$, $\hat{q}_2(X) = \tilde{q}_2(UXV^*)$, and $\hat{r}(X) = \tilde{r}(UXV^*)$ completes the proof.

\end{proof}

To show that $\hat{f}$ is not a hyperwron, we show that restricting $\hat{f}$ to an affine subspace where $\textup{Re}(X_{11}) = \textup{Re}(X_{22}) = 0$ results in a polynomial that cannot be a product of sums of squares.

\begin{theorem}
    The quartic form $\hat{f}$ defined in~\eqref{eq:quartic-eg} is not a hyperwron. 
    \label{theorem: example polynomial is not a hyperwron}
\end{theorem}
\begin{proof}
    We argue by contradiction. Suppose that $\hat{f}$ is a hyperwron. Then, by Lemma~\ref{lemma: diagonalized 16-variable quartic hyperbolic form}, there exist sums of squares $\hat{q}_1,\hat{q}_2\in \Sigma_{16,2}$, a cubic form $\hat{r}\in F_{16,3}$ and real numbers $\sigma_1,\sigma_2\in \R$ such that 
\begin{equation}
\label{eq:hatfrep}
\hat{f}(X) = \hat{q}_1(X)\hat{q}_2(X) + \hat{r}(X)(\sigma_1 \textup{Re}(X_{11}) + \sigma_2 \textup{Re}(X_{22})).
\end{equation}
Let $h$ denote the polynomial in two variables defined by $h(x_1,w_1) = \hat{f}\left(\begin{smallmatrix} x_1 i & i\\i & w_1 i\end{smallmatrix}\right)$. Since we have restricted $\hat{f}$ to an affine space where the diagonal elements have zero real part,  equation~\eqref{eq:hatfrep} tells us that there exist quadratic sums of squares $\rho_1,\rho_2$ such that 
\[ h(x_1,w_1) = \rho_1(x_1,w_1)\rho_2(x_1,w_1).\]
On the other hand we can explicitly see that 
\begin{align*} h(x_1,w_1) &= {\det}_M\left(\begin{pmatrix}
        x_1 i  &i\\ i &w_1i
    \end{pmatrix}\begin{pmatrix}
        x_1 i  &i\\ i &w_1i
    \end{pmatrix}^* \right)\\
    &={\det}_M\begin{pmatrix}
          x_1^2+1 & x_1+w_1\\x_1+w_1 & w_1^2 +1\end{pmatrix}\\
    &=(x_1w_1-1)^2.\end{align*}
Since $h$ is a square, it follows that $\rho_1\rho_2$ is a square, and hence that $\rho_1$ and $\rho_2$ are each squares. 
Therefore $1-x_1w_1 = (a+bx_1+cw_1)(d+ex_1+fw_1)$ for some $a,b,c,d,e,f\in \R$. This implies the following identity on symmetric matrices:
\[ \begin{bmatrix} 1 & 0 & 0\\0 & 0 & -1/2\\0 & -1/2 & 0\end{bmatrix} = \frac{1}{2}\begin{bmatrix} a\\b\\c\end{bmatrix}\begin{matrix}\begin{bmatrix} d & e & f\end{bmatrix}\\\phantom{a}\\\phantom{a}\end{matrix} + \frac{1}{2}\begin{bmatrix} d\\e\\f\end{bmatrix}\begin{matrix}\begin{bmatrix} a & b & c\end{bmatrix}\\\phantom{a}\\\phantom{a}\end{matrix}.\]
This is a contradiction because the left hand size has rank three and the right hand side has rank two. 

\end{proof}

\subsection{The example is extreme}
In this section we show that the quartic form $\hat{f}$, defined in~\eqref{eq:quartic-eg} generates an extreme ray of the cone of non-negative polynomials of degree four in sixteen variables. We do this by applying the following well-known sufficient condition for a form $q\in P_{n,2d}$ to generate an exposed extreme ray (see, e.g.,~\cite[Chapter 4]{2012SemidefiniteGeometry}).
 \begin{proposition}
 \label{prop:suff-extreme}
     Let $q\in P_{n,2d}\setminus\{0\}$ be a non-negative homogeneous polynomial. Let $\mathcal{V}(q) =\{x\in \R^n\;:\; q(x) = 0\}$ and let 
     \begin{equation}\label{eq:Ldef} \mathcal{L}_{q} = \{p\in F_{n,2d}\;:\;\nabla p(x) = 0\;\;\textup{for all $x\in \mathcal{V}(q)$}\}.\end{equation}
     If $\mathcal{L}_q \subseteq \textup{span}(q)$, then $q$ generates an extreme ray of $P_{n,2d}$.
\end{proposition}
\begin{proof}
Let $q = q_1 + q_2$ where $q_1,q_2\in P_{n,2d}$. We will show that if $\mathcal{L}_q \subseteq \textup{span}(q)$, then $q_1$ and $q_2$ are both non-negative multiples of $q$. 

Since $q = q_1+ q_2$ and $q_1,q_2\in P_{n,2d}$, it follows that $q_1(x) = q_2(x) = 0$ for all $x\in \mathcal{V}(q)$. Moreover, since every $x\in \mathcal{V}(q)$ is a global minimizer of $q$ and $q_1$ and $q_2$, it follows that $\nabla q(x) = \nabla q_1(x) = \nabla q_2(x) = 0$ for all $x\in \mathcal{V}(q)$. In other words, $q,q_1,q_2\in \mathcal{L}_q$. Since we have assumed that $\mathcal{L}_q \subseteq  \textup{span}(q)$, it follows that 
there are $\lambda_1,\lambda_2\in \R$ such that $q_1 = \lambda_1\,q$ and $q_2 = \lambda_2\,q$. Since $q,q_1,q_2$ are non-negative, it follows that $\lambda_1\geq 0$ and $\lambda_2\geq 0$. This shows that $q_1$ and $q_2$ are nonnegative multiples of $q$, and so $q$ generates an extreme ray of $P_{n,2d}$.
\end{proof}

\begin{proposition}
 The quartic form $\hat{f}$ defined in~\eqref{eq:quartic-eg} generates an extreme ray of $P_{16,4}$.
 \label{proposition: example polynomial is an extreme ray}
\end{proposition}
\begin{proof}
    By Proposition~\ref{prop:suff-extreme}, it suffices to show that $\mathcal{L}_{\hat{f}} \subseteq \textup{span}(\hat{f})$, where $\mathcal{L}_{\hat{f}}$ is defined in~\eqref{eq:Ldef}. We will first show that  $\hat{f}(X) = 0$ whenever $X\in \mathbb{H}^{2\times 2}$ has rank one.  Note that $X\in \mathbb{H}^{2\times 2}$ is rank one if  \[X=\begin{bmatrix} x\\y\end{bmatrix}\begin{matrix}\begin{bmatrix} z & w\end{bmatrix}\\\phantom{a}\end{matrix}=\begin{bmatrix}x z & xw\\yz & yw\end{bmatrix}\]
for some $x,y,z,w\in \mathbb{H}$. Using this parametric form,  we have
    \begin{align*} \hat{f}\left(\begin{bmatrix}x z & xw\\yz & yw\end{bmatrix}\right) 
    & = {\det}_M\left(\begin{bmatrix} x\\y\end{bmatrix}\begin{matrix}\begin{bmatrix} z & w\end{bmatrix}\\\phantom{a}\end{matrix}\begin{bmatrix} z^*\\w^*\end{bmatrix}\begin{matrix}\begin{bmatrix} x^* & y^*\end{bmatrix}\\\phantom{a}\end{matrix} \right)\\
    & = {\det}_M\left((|z|^2+|w|^2)\begin{bmatrix} x\\y\end{bmatrix}\begin{matrix}\begin{bmatrix} x^* & y^*\end{bmatrix}\\\phantom{a}\end{matrix}\right)\\
    & = (|z|^2+|w|^2)^2(|x|^2|y|^2 - |xy^*|^2)\\
    & = 0.
    \end{align*}
    It follows that
    \[ \mathcal{L}'_{\hat{f}}:= \left\{p\in F_{16,4}\;:\; \nabla p\left(\begin{bmatrix}x z & xw\\yz & yw\end{bmatrix}\right) = 0\;\;\textup{for all $x,y,z,w\in \mathbb{H}$}\right\}\supseteq \mathcal{L}_{\hat{f}}.\]
    To show that $\hat{f}$ generates an extreme ray of $P_{16,4}$, it is enough to show that $\mathcal{L}'_{\hat{f}}\subseteq \textup{span}(\hat{f})$. Since $\hat{f}\in \mathcal{L}'_{\hat{f}}$, it suffices to show that $\dim(\mathcal{L}_{\hat{f}}')\leq \dim(\textup{span}(\hat{f}))=1$. Note that each of the entries of $\nabla p\left(\begin{bmatrix}x z & xw\\yz & yw\end{bmatrix}\right)$ can be thought of as a form of degree $6$ in $16$ real variables with coefficients that are linear in the coefficients of $p\in F_{16,4}$. As such, $\mathcal{L}_{\hat{f}}'$ is the kernel of a linear map $\mathcal{A}:F_{16,4}\rightarrow F_{16,6}^{16}$. More explicitly, we can think of $\mathcal{A}=\mathcal{A}_2\circ \mathcal{A}_1$ as the composition of two linear maps. The linear map $\mathcal{A}_1:F_{16,4}\rightarrow F_{16,3}^{16}$ sends $p$ to $\nabla p$. The linear map $\mathcal{A}_2:F_{16,3}^{16}\rightarrow F_{16,6}^{16}$  sends a tuple $(q_1,\ldots,q_{16})\in F_{16,3}^{16}$ of cubics in a $2\times 2$ matrix of quaternion variables ($16$ real variables) to the corresponding tuple $(r_1,\ldots,r_{16})\in F_{16,6}^{16}$ of sextics in four quaternion variables ($16$ real variables) via the relation
    \[ r_j(x,y,z,w) = q_j\left(\begin{bmatrix}x z & xw\\yz & yw\end{bmatrix}\right)\quad\textup{for $j=1,2,\ldots,16$}.\]

    Let 
    \[\mathcal{U} = \left\{q\in F_{16,3}\;:\; q\left(\begin{bmatrix} xz & xw\\yz & yw\end{bmatrix}\right) = 0\;\textup{for all $x,y,z,w\in \mathbb{H}$}\right\}\]
    denote the subspace of cubic forms in 16 variables that vanish on rank one $2\times 2$ quaternionic matrices. Then
    $\mathcal{U}^{16}\subseteq F_{16,3}^{16}$ is the kernel of the linear map $\mathcal{A}_2$.
    We can rewrite the subspace of interest as
    \[ \mathcal{L}_{\hat{f}}' = \{p\in F_{16,4}\;:\; \nabla p \in \mathcal{U}^{16}\}.\]
    Recall that our aim is
    to show that the dimension of $\mathcal{L}_{\hat{f}}'$ is at most one. 
    
Next, we construct an explicit basis for $\mathcal{U}$. Since $\hat{f}$ is a non-negative quartic form that vanishes on rank one $2\times 2$ quaternionic matrices, it follows that each partial derivative of $\hat{f}$ (i.e., each entry of the gradient of $\hat{f}$) is an element of $\mathcal{U}$. Moreover, the $16$ partial derivatives of $\hat{f}$ are linearly independent (which can be confirmed by noting that the Hessian of $\hat{f}$ evaluated at $x=z=1,y=w=0$ has full rank). To see
that these span $\mathcal{U}$, we form a (sparse, integer-valued) matrix with $\mathcal{U}$ as its nullspace, and use this to explicitly compute that the dimension of $\mathcal{U}$ is $16$. This
confirms that $\mathcal{U}$
has the $16$ partial derivatives of $\hat{f}$ as a basis.

It follows that $p\in \mathcal{L}_{\hat{f}}'$ if and only if
there exists a $16\times 16$ real matrix $A$ such that $\nabla p(X) = A\nabla \hat{f}(X)$ for all $X\in \mathbb{H}^{2\times 2}$.

    Consider the subspace 
    \[ \tilde{\mathcal{L}}_{\hat{f}} = \{(p,A)\in F_{16,4}\times \R^{16\times 16}\;:\; \nabla p(X) = A \nabla \hat{f}(X)\;\;\textup{for all $X\in \mathbb{H}^{2\times 2}\cong \R^{16}$}\}\]
    and note that $\mathcal{L}_{\hat{f}}'$
    is the image of $\tilde{\mathcal{L}}_{\hat{f}}$ under the surjective linear map $(p,A)\mapsto p$. As such, to show that $\dim(\mathcal{L}_{\hat{f}}')\leq 1$, it is enough to show that
    $\tilde{\mathcal{L}}_{\hat{f}}$ is one-dimensional. The subspace $\tilde{\mathcal{L}}_{\hat{f}}$ is, again, the kernel of the linear map $\mathcal{B}:F_{16,4}\times \R^{16 \times 16} \rightarrow F_{16,3}^{16}$ defined by $\mathcal{B}(p,A) = \mathcal{A}_1(p) - A\nabla \hat{f}$. 
 Directly forming the corresponding (sparse integer-valued) matrix that represents $\mathcal{B}$ with respect to the monomial basis and computing the dimension of its nullspace reveals that $\tilde{\mathcal{L}}_{\hat{f}}$ has dimension one and therefore that $\mathcal{L}_{\hat{f}}'$ has dimension at most one. This completes the proof.
    
    Mathematica code that sets up matrices with nullspaces $\mathcal{U}$ and $\tilde{\mathcal{L}}_{\hat{f}}$, and computes the respective dimensions of these nullspaces, can be found at \href{https://github.com/BrianNg001/Non-negative-polynomials-without-hyperbolic-certificates-of-non-negativity/}{https://github.com/BrianNg001/Non-negative-polynomials-without-hyperbolic-certificates-of-non-negativity/}.
\end{proof}
We are now in a position to state and prove the main result of this section.

\begin{proof}[{Proof of Theorem~\ref{thm:main-sum-hyperwrons}}]
        We argue by contradiction. If $\hat{f} = \sum_{i=1}^k f_i$ were a sum of hyperwrons $f_i\in \mathcal{W}_{16,4}$ (for $i=1,2,\ldots,k$), then $\hat{f}$ would be a sum of nonnegative forms (since every hyperwron is nonnegative). Since $\hat{f}$ generates an extreme ray of $P_{16,4}$ by Proposition~\ref{proposition: example polynomial is an extreme ray}, it follows that all of the $f_i$ are non-negative multiples of $\hat{f}$. But then $\hat{f}$ must be a hyperwron, which contradicts Theorem~\ref{theorem: example polynomial is not a hyperwron}. Therefore $\hat{f}$ is not a sum of hyperwrons.
\end{proof}

%% file: Main_Body/Research_Outlook.tex
\section{Discussion}
\label{section: Discussions}

This paper considers the question of whether all non-negative polynomials can be expressed as hyperwrons, hyperzouts, or sums of these. We show that there are non-negative polynomials that are not hyperwrons, and give an explicit example of a quartic form that is not a sum of hyperwrons. Our techniques do not give such strong results in the case of hyperzouts, however. We establish that if we restrict the degree of the hyperbolic polynomial that forms part of the construction of a hyperzout, then there are non-negative polynomials that are not hyperzouts. However, this does not rule out the possibility that every non-negative polynomial is a hyperzout.

It is natural to ask whether the result in Theorem~\ref{theorem: m_y no Wronskian} can be improved, in the sense that there are additional cases of degrees and numbers of variables where there exist non-negative homogeneous polynomials that are not hyperwrons. The cases that are not settled are:
\begin{itemize}
     \item $m=3$ and $y\geq 3$ (ternary forms of degree at least six);
     \item $m=4$ and $y=\{2,3\}$ (quaternary forms of degree four and six);
     \item $m=5$ and $y=2$ (quartic forms in five variables).
\end{itemize}
The dimension count in the proof of Theorem~\ref{theorem: m_y no Wronskian} could actually be sharpened slightly. 
For instance, we over-count dimensions because we do not exploit certain scaling symmetries in the map $\Theta_1$. There may be other opportunities to refine this argument to sharpen the result in Theorem~\ref{theorem: m_y no Wronskian}.

To make further progress, a deeper understanding of the properties and, in particular, the zeros of hyperwrons and hyperzouts, is required. Our more refined results in Section~\ref{section: An example non-negative quartic homogeneous polynomial that is not hyperwron}, for example, show that every degree four hyperwron decomposes as a product of two sums of squares upon restriction to a suitable codimension one subspace. This allows us to construct an example of a quartic form that is not a hyperwron. A natural approach to showing that there exist quartic forms that are not hyperzouts would be to seek analogous properties that hold for all quartic hyperzouts.

In this work, we have made some progress in understanding the relationship between polynomials with certain hyperbolic certificates of non-negativity and the full cone of non-negative polynomials. It is natural to attempt to understand the relationships between hyperwrons (or hyperzouts) and other families of non-negative polynomials, such as sums of non-negative circuit polynomials~\cite{Iliman2016LowerProgramming,Dressler2017APolynomials}. For instance, Blekherman et al.~\cite[Theorem 6.3]{Blekherman2023LinearContainment} present a quartic homogeneous polynomial that is both a hyperwron and a sum of non-negative circuit polynomials but is not a sum of squares. In the spirit of the present paper, one could ask whether all sums of non-negative circuit polynomials are (sums of) hyperwrons.

\subsection{Extension to non-negative polynomials from interlacers}
Let $p\in \textup{Hyp}_{n,d}(e)$ be hyperbolic with respect to $e$. We say that $q\in F_{n,d-1}$ \emph{interlaces $p$ with respect to $e$} if the roots of the univariate polynomials $t\mapsto p(te-x)$ and $t\mapsto q(te-x)$ interlace for all $x\in \R^n$. More explicitly, this means that if $\lambda_1(x)\leq \cdots \leq \lambda_d(x)$ are the roots of $p(te-x)$ and $\mu_1(x),\mu_2(x),\cdots, \mu_{d-1}(x)$ are the roots of $q(te-x)$ then 
\[\lambda_1(x) \leq \mu_1(x) \leq \lambda_2(x) \leq \cdots \leq \lambda_{d-1}(x) \leq \mu_{d-1}(x) \leq \lambda_d(x).\] 
If $q$ interlaces $p$ with respect to $e$ then $q$ is necessarily hyperbolic with respect to $e$.

It is known (see~\cite[Theorem 2.1]{Kummer2012HyperbolicSquares}) that if $q$ interlaces $p$ with respect to $e$ then $D_ep(x)q(x) - D_eq(x)p(x) \geq 0$ for all $x\in \R^n$. As such, for a fixed $p\in \textup{Hyp}_{n,d}(e)$, we can generate non-negative polynomials by taking any $q$ that interlaces $p$ with respect to $e$, and any polynomial map $\phi$, and considering polynomials of the form 
\begin{equation}
    \label{eq:interlace-cert}
D_ep(\phi(x))q(\phi(x)) - D_ep(\phi(x))p(\phi(x)).
\end{equation}
It is clear that if $p$ has degree two then $D_ep\,q - D_eq\, p$ is a sum of squares, since it has degree two and is non-negative. Therefore, in the case $d=2$, any expression of the form $D_ep(\phi(x))q(\phi(x)) - D_ep(\phi(x))p(\phi(x))$ (where $q$ interlaces $p$ with respect to $e$) is a sum of squares.

One could then consider whether every non-negative polynomial can be expressed in the form~\eqref{eq:interlace-cert}, for some interlacing pair $p$ and $q$ and polynomial map $\phi$. The argument in Theorem~\ref{theorem: Dimension bound of Wronskian from hyperbolic polynomial} directly extends to this setting. Indeed one can show that under the same assumptions on the number of variables $m$ and the degree $2y$ as in Theorem~\ref{theorem: m_y no Wronskian}, there are non-negative polynomials $f\in P_{m,2y}$ that can not be expressed in the form~\eqref{eq:interlace-cert} where $p$ and $q$ are an interlacing pair and $\phi$ is a polynomial map.